\documentclass[preprint,3p,times]{elsarticle}
\usepackage[toc,page,title,titletoc,header]{appendix}
\usepackage{amsmath}
\usepackage{amssymb}
\usepackage{amsthm}
\usepackage{enumerate}
\usepackage{enumitem}
\usepackage{mathrsfs}
\usepackage{graphicx}
\usepackage{epsfig}
\usepackage{tikz,pgfplots,tikz-3dplot}
\usepackage{epstopdf}
\usepackage{microtype}
\usetikzlibrary{fit}
\biboptions{sort&compress}

\newtheorem{thm}{Theorem}[section]
\newtheorem{lem}[thm]{Lemma}
\newtheorem{rem}{Remark}
\renewcommand{}

\newcommand{\vol}{\operatorname{vol}}

\newcommand{\bI}{\mathbb{I}}

\newcommand{\nn}{\nonumber}

\newcommand{\ttau}{\Delta t}
\newcommand{\mS}{\mathcal{S}}

\def\epsilon{\varepsilon} 
\newcommand{\mat}[1]{\boldsymbol{#1}}

\allowdisplaybreaks
\newcommand{\pol}{\vec}  
\begin{document}
\begin{frontmatter}
\title{
Structure-preserving parametric finite element methods for simulating axisymmetric solid-state dewetting problems with anisotropic surface energies
}

\author[1]{Meng Li}
\address[1]{School of Mathematics and Statistics, Zhengzhou University,
Zhengzhou 450001, China}
\ead{limeng@zzu.edu.cn}
\author[1]{Chunjie Zhou}

\begin{abstract}
Solid-state dewetting (SSD), a widespread phenomenon in solid-solid-vapor system, could be used to describe the accumulation of solid thin films on the substrate. 
In this work, we consider the sharp interface model for axisymmetric SSD with anisotropic surface energy. 
By introducing two types of surface energy matrices from the anisotropy functions,
we aim to design two structure-preserving algorithms for the axisymmetric SSD. The newly designed schemes are applicable to a broader range of anisotropy functions, and we can theoretically prove their volume conservation and energy stability. 
In addition, based on a novel weak formulation for the axisymmetric SSD, we further build another two numerical schemes that have good mesh properties. Finally, numerous numerical tests are reported to showcase the accuracy and efficiency of the numerical methods. 

\end{abstract} 


\end{frontmatter}

\setcounter{equation}{0}
\section{Introduction} \label{sec:intro}

At temperatures considerably lower than the material’s melting point, solid thin films on substrates tend to become unstable, prompting either dewetting or agglomeration, ultimately forming isolated islands.
Solid films remain in a solid state throughout their evolution, hence the term “solid-state dewetting (SSD)” is used to describe this process \cite{Thompson2012SolidStateDO,Leroy16}.
SSD is a widespread phenomenon in nature, primarily utilized in materials science and physics. It occurs within solid-solid-vapor systems, describing the process of solid thin films agglomerating on a substrate. The evolution of solid films under the influence of surface tension and capillary effects often showcases intricate characteristics. These include phenomena such as faceting \cite{Mccallum96, Jiran92,Ye10a}, edge retraction \cite{Wong00,Dornel06,hyun2013quantitative} caused by the reduction of surface curvature gradient, and fingering instabilities \cite{Kan05,Ye10b,Ye11a,Ye11b}. 

Recently, SSD finds extensive applications in numerous modern technologies. For instance, the SSD of thin films, posing an essential challenge in microelectronics processing, can be utilized to produce well-controlled patterns of micro-/nanoscale particle arrays.
These arrays find great applications in various fields such as sensors \cite{Mizsei93}, optical and magnetic devices \cite{Armelao06}, as well as catalysts for the growth of carbon and semiconductor nanowires \cite{Schmidt09}. 
The significant industrial applications and scientific inquiries surrounding SSD inspire many researchers to delve into understanding its underlying mechanisms, including both the experimental \cite{Amram12,Rabkin14,Herz216,Naffouti16,Naffouti17,Kovalenko17} and theoretical \cite{Wang15,Jiang16,Bao17,Bao17b,Zucker16} efforts.

In general, kinetic process of the evolution of films are governed by surface diffusion flow and contact line migration. Srolovitz and Safran \cite{Srolovitz86a} first proposed an isotropic sharp-interface model with small slope profile and cylindrical symmetry for simulating hole growth. 
The work was further developed to both the 2-dimensional \cite{Wong00} and 3-dimensional \cite{Du10} cases in Lagrangian representation.
Then, “marker particle” numerical method was designed by Wong et al \cite{Wong00} for solving nonlinear isotropic sharp-interface model without the assumption of small slope.
A phase-field model was designed by Jiang et al. \cite{Jiang12} to simulate the SSD of thin films with isotropic surface energy.
However, a significant number of experiments have demonstrated that crystalline anisotropy has a substantial impact on the kinetic evolution during SSD \cite{Thompson2012SolidStateDO,Leroy16}. Recently, many approaches have been proposed to study the effect of surface energy anisotropy, such as the discrete model by Dornel \cite{Dornel06}, the kinetic Monte Carlo model \cite{Dufay11,Pierre09a}, and the models by crystalline method \cite{hyun2013quantitative,Zucker13}. 
Additionally, a phase field approach for SSD problems with weakly anisotropic surface energy was studied in \cite{Wang15}. The method can inherently capture the intricate topology changes during evolution.
Moreover, comprehensive studies have been conducted on 2-dimensional SSD problems using sharp-interface models \cite{Jiang19a,Jiang19c,Zhao19}. In contrast to other approaches, these models are meticulously derived using the energetic variation method, enabling seamless integration of anisotropy and providing a fully mathematical representation. 
The governing equations for the SSD fall into a type of fourth-order (for weak anisotropy) or sixth-order (for strong anisotropy) geometric partial differential equations (PDEs) with prescribed boundary conditions at the two contact points.

Parametric finite element methods (PFEMs) have been widely regarded as highly effective approaches for solving geometric PDEs, with many advantages over other methods, such as weaker restriction on the time step and better mesh distribution, see the isotropic cases in \cite{Bansch05,bao2021structure,Barrett07,Barrett08JCP,kovacs2021convergent,Zhao20} and the general anisotropic cases in \cite{Bao17,baojcm2022,Barrett07Ani,Barrett08Ani,Hausser07,li2021energy,Zhao19b,Barrett20}. 
Among the PFEMs, the 'BGN' method (introduced by Barrett, Garcke and N{\"urnberg} in \cite{Barrett08JCP}) is regarded as an effective and prominent approach, as it allows for tangential degrees of freedom, ensuring excellent mesh quality. This also eliminates the need for the mesh regularization/smoothing procedures commonly required in numerous other methods. 
We refer to the review article \cite{Barrett20} for more thorough understanding on this idea. 
Very recently, an energy-stable PFEM for the surface diffusion flow and the SSD
with weakly anisotropic surface energy was proposed in \cite{li2021energy}. However, it has some relative complicated limitations on the anisotropic surface energy density $\gamma(\theta)$ with $\theta$ the angle between the outward unit normal vector and the vertical axis. Later, Bao et al. \cite{bao2023symmetrized,bao2023symmetrized1} constructed symmetrized energy-stable PFEMs for the 2- and 3-dimensional surface diffusion flows with symmetric surface energy density (related to the normal vector $\bf n_\mS$), i.e. $\gamma(-\bf n_\mS)=\gamma(\bf n_\mS)$. This method was also applied in the SSD problem, see \cite{li2023symmetrized}. 
In \cite{bao2024structure}, novel energy-stable PFEMs were proposed for some types of 2-dimensional anisotropic flows with a mild condition: $\gamma(-{\bf n_\mS})< 3\gamma({\bf n_\mS})$. In \cite{bao2023unified}, a unified structure-preserving PFEM for anisotropic surface diffusion in two dimensions and three dimensions was established under the condition $\gamma(-{\bf {n_\mS}})< (5-d)\gamma({\bf {n_\mS}})$.

In this work, we focus on the structure-preserving algorithms for the SSD with axisymmetric geometry. Indeed, 
if the evolving 3-dimensional surface has rotational symmetry structure, we can reduce the
geometric flows into the 1-dimensional simple problems. This treatment can significally minimize the computational complexity, avoid intricate mesh controls by dealing with the 1-dimensional generating curve, and maintain the axisymmetric property throughout the evolutionary process.
Zhao \cite{Zhao19} proposed a sharp-interface model for simulating SSD with axisymmetric geometry based on thermodynamic variation. Then a PFEM was proposed to solve above sharp-interface model. However, the numerical method is not structure-preserving, including both volume-conservative and energy-stable. 
We in this work review this system, and aim to establish its structure-preserving algorithms. Different from  
\cite{bao2023symmetrized,bao2023symmetrized1} and motivated by \cite{li2021energy}, we introduce two types of surface energy matrices with related to the variable $\theta$, and then build the equivalent systems of the sharp-interface model. Meanwhile, the surface energy matrices in this article are different from one in \cite{li2021energy}, as we add a stabilized term in each matrix in order 
to derive the energy stability of the PFEM. 
We also notice that for this special axisymmetric SSD problem in three dimensions,  the energy-stable condition $\gamma(\theta+\pi)< 2\gamma(\theta)$ can be weaken into $\gamma(\theta+\pi)< 3\gamma(\theta)$.

In summary, the primary objectives of this article include: (i) introduce two novel forms of surface energy matrices, and obtain the equivalent systems of the sharp-interface model; (ii) build two types of weak formulations and then establish three types of PFEMs with different properties, including structure-preserving approximation, linear approximation and volume-preserving approximation;  
(iii) present some numerical examples to test convergence rates, mesh quality, structure-preserving properties of the proposed PFEMs and investigate some new kinetic processes of SSD during its evolution process.

The rest of the paper is organized as follows. In Section \ref{SEC:MATHF}, we recall the sharp-interface model for SSD with axisymmetric geometry. In Section \ref{sec3}, we present a unified surface energy matrix and derive a novel variational formula, demonstrating the volume conservation and energy dissipation of the continuous model. In Section \ref{sec4}, structure-preserving PFEMs are proposed. In Section \ref{sec5}, we propose two novel PFEMs that enhance the quality of the mesh. In Section \ref{sec6}, a large number of numerical tests are conducted to demonstrate the validity of the proposed theory. Finally, we come to some conclusions in Section \ref{sec7}. 




\section{The sharp-interface model}\label{SEC:MATHF}
\setcounter{equation}{0}

In this section, we first review the SSD with axisymmetric geometry \cite{Zhao19}. As depicted in Fig. \ref{fig:det} (a), a toroidal thin island film is positioned on a flat and rigid substrate, with the generatrix illustrated in Fig. \ref{fig:det} (b) \cite{Zhao19}.
\begin{figure}[!htp]
\centering
\includegraphics[width=0.8\textwidth]{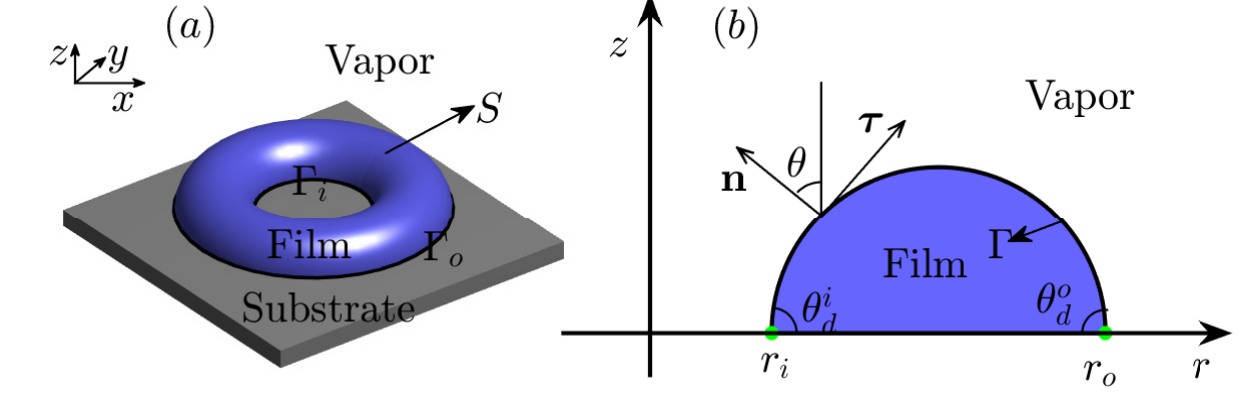}
\caption{A schematic description of the SSD: (a) a toroidal thin film on a flat substrate; (b) the cross-section of an axis-symmetric thin film in the cylindrical coordinate system $(r, z)$. 
Here, $r_i$ and $r_o$ represent the radius of the inner and outer contact lines, respectively.}
\label{fig:det}
\end{figure}

The thin film is characterized by an open surface $\mathcal{S}$, with its boundaries identified as two closed curves $\Gamma_i$ and $\Gamma_0$ situated on the substrate.
Since the graph enclosed by $\mS$ exhibits axisymmetry, we can parameterize the open surface $\mS$ as follows
\begin{equation}
(s,\varphi)\to \mS(s,\varphi):=\left(r(s)\cos\varphi,r(s)\sin\varphi,z(s)\right),
\end{equation}
where $r(s)$ is the radial distance, $\varphi$ represents azimuth angle, $z(s)$ is the film height, and $s$ $\in$ $[0, L]$ represents the arc length along the radial direction curve. $\mS(0,\cdot)$ and $\mS(L,\cdot)$ represent the inner contact line $\Gamma_i$ and outer contact line $\Gamma_0$, respectively.

The total interface energy for SSD problem can be written as 
\begin{equation}\label{eqn:energy}
W =\iint\limits_{\mS}\gamma_{FV} (\vec n)\,d\mS+\underbrace{(\gamma_{FS}-\gamma_{VS})A(\Gamma_o/\Gamma_i)}_{\text{Substrate\, energy}}\,,
\end{equation}
where $A(\Gamma_o/\Gamma_i)$ represents the surface area is surrounded by the two contact lines on the subtrate, $\gamma_{FS}$ and $\gamma_{VS}$ denote the surface energy densities of film/substrate and vapor/substrate respectively, and $\gamma_{FV}(\vec n)$ represents surface energy density of the film/vapor (surface $\mS$) with $\vec n$ the unit normal vector.

Since the cylindrical symmetry can reduce its dependence on the orientation of curve in the radial direction,  $\gamma(\theta)=\gamma_{FV}(\vec n)$ can represent the surface energy density of film/vapor satisfying 
\begin{equation}
\theta=\arctan\frac{z_s}{r_s};\qquad\gamma(\theta)=\gamma(-\theta),\quad\forall\theta\in [0,\pi];\qquad\gamma(\theta)\in C^2([0,\pi]),
\end{equation}
where subscript $s$ means the derivative of $s$.
We assume $r_o$ and $r_i$ represent the radius of outer contact line and inner contact lines of film/vapor on the substrate. To simplify, we let $r_o=r_0$, $r_i=r_L$. Consequently, the total interface energy can be simplified as:
\begin{equation}
W =\iint\limits_{\mS}\gamma (\theta)\,d\mS+\underbrace{(\gamma_{FS}-\gamma_{VS})(\pi r_0^2-\pi r_L^2)}_{\text{Substrate\, energy}}.
\end{equation}

We denote $\Gamma(t)$ by the generatrix of the open surface $\mathcal{S}$, given by $\vec{X}(s,t) = \left( r(s,t), z(s,t) \right)^\top$, with $r$ and $z$ representing the functions of the arc length $s$ and the time $t$. Then, 
the sharp interface model for SSD with anisotropic surface energy in three dimensions with cylindrical symmetry can be obtained as the following dimensionless form:
\begin{subequations}\label{eqn:model}
\begin{align}
&\partial_t \vec{X} =-\frac{1}{r}\partial_s(r\partial_s\mu)\vec{n},\quad 0<s<L(t),\quad t>0, \label{eqn:model_a}\\
&\mu = -\left[\gamma(\theta)+\gamma''(\theta)\right]\kappa+\frac{\gamma(\theta)\partial_s z + \gamma'(\theta)\partial_s r}
{r}, \label{eqn:model_b}\\
&\kappa =-(\partial_{ss}\vec{X})\cdot\vec{n},\qquad\vec{n}=-(\partial_s\vec{X})^\perp,
\label{eqn:model_c}
\end{align}
\end{subequations}
where $\mu$ represents 
the chemical potential, $\kappa$ denotes the curvature of the open curve $\Gamma(t)$, and $L(t)$ is 
the total arc length of the open curve $\Gamma(t)$. 
The initial curve is given as 
\begin{equation}\label{eqn:initial}	
\vec{X} (s,0):=\vec{X}_0 (s)=(r(s,0), z(s,0))^\top=(r_0 (s),z_0 (s))^\top,\qquad0\le s\le L_0:=L(0).
\end{equation}
The governing equation mentioned satisfies the specified boundary conditions:
\begin{itemize}
\item [(i)] contact line condition
\begin{equation}
\label{eqn:BC1}
z(L,t)=0,\qquad \begin{cases} 
z(0,t)=0, & \text{if}~~r(0,t)>0, \\
\partial_s z(0,t)=0, & \text{otherwise}, 
\end{cases}\qquad t\ge0;
\end{equation}
\item [(ii)] relaxed contact angle condition
\begin{equation}
\label{eqn:BC2}
\partial_t r(L,t)=-\eta f(\theta_d ^0;\sigma),\qquad \begin{cases} 
\partial_t r(L,t)=-\eta f(\theta_d ^i;\sigma), & \text{if}~~ r(0,t)>0, \\
r(0,t)=0, & \text{otherwise}, 
\end{cases}\qquad t\ge0;
\end{equation}
\item [(iii)] zero-mass flux condition 
\begin{equation}
\label{eqn:BC3}
\partial_s\mu (0,t)=0, \qquad \partial_s\mu(L,t)=0, \qquad t\ge0,
\end{equation}
\end{itemize}
where $\theta_d ^0$ and $\theta_d ^i$ 
 denote the angles at the right and left contact lines respectively, $\eta\in(0,\infty)$ is the contact line mobility, and $f(\theta;\sigma)$ is defined as follows:
\begin{equation}\label{eqn:f}
f(\theta;\sigma)=\gamma(\theta)\cos(\theta)-\gamma'(\theta)\sin(\theta)-\sigma, \quad\theta\in[-\pi,\pi], \quad\sigma=\frac{\gamma_{VS}-\gamma_{FS}}{\gamma_0}. 
\end{equation}
\begin{rem}\label{rem:boundary}
The contact line condition (i) guarantees the continuous movement of contact lines along the substrate.  The relaxed contact angle condition (ii) permits the adjustment of the contact angle, while the zero-mass flux condition (iii) maintains the conservation of the total volume/mass of the thin film.  
\end{rem}

Define $\text{vol}(\vec{X}(t))$ as the volume enclosed between the surface $\mathcal{S}$ and the substrate, and let $W(t)$ be the total free energy. By using the surface integral calculation, we have 
\begin{equation}
\vol(\vec{X}(t))=2\pi\int_{0}^{L(t)} rzr_s \,ds, \quad
W(t)=2\pi\int_{0}^{L(t)} r\,\gamma(\theta)\,\left | \partial_s\vec X \right |ds - \sigma\pi(r_0 ^2-r_L ^2). 
\end{equation}
From \cite{Zhao19}, the following volume conservation and energy decay properties hold 
\begin{align}
  \vol(\vec{X}(t))\equiv \vol(\vec{X}(0)); 
  \qquad W(t_2)\leq W(t_1)\leq W(0),
  \qquad t_2\geq t_1\geq 0.
\end{align}

\section{Variational formulations and properties}\label{sec3}
In this section, we will introduce a  variational formulation of the model \eqref{eqn:model}, and demonstrate its volume-conservation and energy-dissipation properties. 
We first introduce the time-independent variable $\rho\in\mathbb{I}=[0,1]$ utilized to parameterize the open curve $\vec{X}(t)$ as follows 
\begin{equation}\label{qx}
\Gamma(t)=\vec{X}(\rho,t)=\left( r(s,t), z(s,t) \right)^\top:\mathbb{I}\times[0,T]\to\mathbb{R}^2. 
\end{equation}
Due to this parametrization, we can obtain the relationship between $s$ and $\rho$ as $s(\rho,t)=\int_{0}^{\rho} | \partial_\rho\vec{X} |d\rho $. Furthermore, we can also obtain $\partial_\rho s=| \partial_\rho\vec{X}  |$ and $ds=\partial_\rho s d\rho= | \partial_\rho\vec{X}  |d\rho$. 

Next, we define the functional space on the domain $\mathbb{I}$ as 
\begin{equation*}
L^2(\mathbb{I}):=\left \{ u:\mathbb{I}\to\mathbb{R}\,\bigg|\,\int\limits_{\Gamma(t)}\left | u(s) \right |^2ds=\int\limits_\mathbb{I}\left | u(s(\rho,t)) \right |^2\partial_\rho s\,d\rho<+\infty \right \},  
\end{equation*}
equipped with the $L^2$-inner product 
\begin{equation*}
(u,v):=\int\limits_{\Gamma(t)}u(s)\,v(s)ds=\int\limits_{\mathbb{I}}u(s(\rho,t))\,v(s(\rho,t))\,\partial_\rho s\,d\rho, \quad\forall u,v\in L^2(\mathbb{I}). 
\end{equation*}
We can directly extend the above inner product to $[L^2(\mathbb{I})]^2$. We further define the Sobolev spaces 
\begin{align*}
&H^1(\mathbb{I}):=\left \{ u:\mathbb{I}\to\mathbb{R},u\in L^2(\mathbb{I})\nn\
\text{and}\ \partial_\rho u \in L^2(\mathbb{I}) \right \}, \\
&H_0^1(\mathbb{I}):=\left \{ u:\mathbb{I}\to\mathbb{R},u\in H^1(\mathbb{I})\nn\
\text{and}\ u(0)=u(1)=0 \right \}, \quad\mathbb{X}:=H^1(\mathbb{I})\times H_0^1(\mathbb{I}).
\end{align*}

We introduce a matrix $\mat{B}_q(\theta)$ related to $\theta$,
\begin{equation}\label{Matrix:Bqx}
\mat{B}_{q}(\theta)=
\begin{pmatrix}
\gamma(\theta)&-\gamma'(\theta) \\
\gamma'(\theta)&\gamma(\theta)
\end{pmatrix}
\begin{pmatrix}
\cos2\theta&\sin2\theta \\
\sin2\theta&-\cos2\theta
\end{pmatrix}^{1-q}+
\mathscr S(\theta)\left[\frac{1}{2}\mat{I}-\frac{1}{2}\begin{pmatrix}
\cos2\theta&\sin2\theta \\
\sin2\theta&-\cos2\theta
\end{pmatrix}\right], 
\end{equation}
where the variable $q$ takes values of 0 and 1, $\mat I$ is a $2\times 2$ identity matrix, and $\mathscr S(\theta)$ is a stability function.  
In this work, the following two cases will be considered:
    \begin{itemize}
        \item [\textbf{(1)}] In the case of $q=0$, we have $\mat{B}_0(\theta)=\mat{B}_0(\theta)^\top$, i.e., the matrix is symmetric. In order to 
   ensure the positive definiteness of $\mat{B}_0(\theta)$, it requires $\gamma(\theta)$ to satisfy $\gamma(\theta)=\gamma(\pi+\theta)$.
        \item [\textbf{(2)}] In the case of $q=1$, the matrix is not symmetric, and it can be split into the symmetric positive matrix $\overline{\mat{B}}_1(\theta)$ and anti-symmetric matrix $\underline{\mat{B}}_1(\theta)$:
        \begin{align*}
       \overline{\mat{B}}_1(\theta)    =\begin{pmatrix}
\gamma(\theta)&0 \\
0&\gamma(\theta)
\end{pmatrix}+
\mathscr S(\theta)\begin{pmatrix}
\frac{1-\cos2\theta}{2}&-\frac{1}{2}\sin2\theta \\
-\frac{1}{2}\sin2\theta&\frac{1+\cos2\theta}{2}
\end{pmatrix},
\qquad 
\underline{\mat{B}}_1(\theta)    =\begin{pmatrix}
0&-\gamma'(\theta) \\
\gamma'(\theta)&0
\end{pmatrix}.
        \end{align*}
To prove the energy stability of the numerical scheme, we also assume  
 $3\gamma(\theta)> \gamma(\pi+\theta)$ that will be introduced later. 
    \end{itemize}

We denote $\vec\tau$ as the tangent vector of the open curve, and $\vec{n}$ as its unit normal vector. Then, there hold
\begin{equation}\label{tangent}
\vec\tau=\partial_s\vec{X}=(\cos\theta,\sin\theta)^\top,\quad \vec{n}=-\vec{\tau}^\bot, \quad \vec n_s=-\partial_{ss}\vec{X}^\bot, \quad 
\partial_s\theta=(\sin^2\theta+\cos^2\theta)\partial_s\theta=\partial_{ss}\vec X\cdot\vec n.
\end{equation}
  Utilizing the matrix $\mat{B}_q(\theta)$ as defined in  \eqref{Matrix:Bqx}, we can obtain the following lemma. 
\begin{lem}\label{lem:equivalent}
 With the matrix $\mat{B}_q(\theta)$, the equation \eqref{eqn:model_b} can be written as 
\begin{equation}
\label{eqn:equiv}
r\mu\vec{n}=\partial_s\left[r\mat{B}_q(\theta)\partial_s\vec{X}\right]-\gamma(\theta)\vec e_1,
\qquad \text{with}\quad \vec e_1=(1, 0)^\top.
\end{equation} 
\end{lem}

\begin{proof}
From \eqref{qx}  and \eqref{tangent},  we can easily obtain 
\begin{equation}\label{eqn:equiv_pf1}
z=\vec{X}\cdot\vec e_2, \quad r=\vec{X}\cdot\vec e_1, \quad \partial_s z =\partial_s\vec{X}\cdot\vec e_2, \quad 
\partial_s r =\partial_s \vec{X}\cdot\vec e_1, 
\end{equation}
where $e_2=(0,1)^\top$. 
By substituting equations \eqref{eqn:equiv_pf1} and \eqref{eqn:model_c} into \eqref{eqn:model_b}, and utilizing $\gamma(\theta)=-\vec\xi\cdot\vec\tau$, we have 
\begin{align}
\label{eqn:equiv_pf2}
r[\gamma(\theta)+\gamma''(\theta)](\partial_{ss}\vec X\cdot\vec n)\vec n&=r\left[\gamma(\theta)(\partial_{ss}\vec X\cdot\vec n )\vec n+ 
\gamma''(\theta)(\partial_{ss}\vec X\cdot\vec n )\vec n+\gamma'(\theta)(\partial_{ss}\vec X\cdot\vec n)(\vec\tau\cdot\vec n)\vec n +\gamma'(\theta)(\partial_{ss}\vec X\cdot\vec n)(\vec\tau\cdot\vec n)\vec n\right] \nonumber\\
&=r\left[\gamma(\theta)\vec\tau_s + \gamma''(\theta)\partial_s\theta \vec n+\gamma'(\theta)\partial_s\theta \vec\tau+\gamma'(\theta)\partial_s\vec n\right] \nonumber\\
&=r\partial_s\left[\gamma(\theta)\vec\tau +\gamma'(\theta)\vec n\right] \nonumber \\
&=\partial_s\left[r\gamma(\theta)\vec\tau +r\gamma'(\theta)\vec n\right]-\gamma(\theta)(\vec\tau\cdot\vec e_1)\vec\tau-(\vec\tau\cdot\vec e_1)\gamma'(\theta)\vec n .
\end{align}
From the definition of $\mat{B}_{q}(\theta)$, it holds 
\begin{align}
\label{eqn:equiv_pf3}
\mat{B}_{q}(\theta)\partial_s\vec X&=
\begin{pmatrix}
\gamma(\theta)&-\gamma'(\theta) \\
\gamma'(\theta)&\gamma(\theta)
\end{pmatrix}
\begin{pmatrix}
\cos2\theta&\sin2\theta \\
\sin2\theta&-\cos2\theta
\end{pmatrix}^{1-q}\binom{\cos\theta}{\sin\theta}+
\mathscr S(\theta)\begin{pmatrix}
\frac{1-\cos2\theta}{2}&-\frac{1}{2}\sin2\theta \\
-\frac{1}{2}\sin2\theta&\frac{1+\cos2\theta}{2}
\end{pmatrix}\binom{\cos\theta}{\sin\theta} \nonumber \\
&=\begin{pmatrix}
\gamma(\theta)&-\gamma'(\theta) \\
\gamma'(\theta)&\gamma(\theta)
\end{pmatrix}\binom{\cos\theta}{\sin\theta} =\gamma(\theta)\binom{\cos\theta}{\sin\theta} +\gamma'(\theta)\binom{-\sin\theta}{\cos\theta}\nonumber \\
&
=\gamma(\theta)\vec\tau+\gamma'(\theta)\vec n.
\end{align}
Hence, by using \eqref{eqn:equiv_pf2} and \eqref{eqn:equiv_pf3}, it follows that  
\begin{align}
\label{eqn:equiv_pf4}
r[\gamma(\theta)+\gamma''(\theta)](\partial_{ss}\vec X\cdot\vec n)\vec n
=\partial_s\left[r\mat{B}_{q}(\theta)\partial_s\vec X\right]-\gamma(\theta)(\vec\tau\cdot\vec e_1)\vec\tau-(\vec\tau\cdot\vec e_1)\gamma'(\theta)\vec n.
\end{align}
In addition, we have 
\begin{align}
\label{eqn:equiv_pf5}
[\gamma(\theta)\partial_s z + \gamma'(\theta)\partial_s r]\vec n&=\left[\gamma(\theta)\partial_s\vec X\cdot\vec e_2+\gamma'(\theta)(\partial_s \vec X\cdot\vec e_1)\right]\vec n \nonumber \\
&=[\gamma(\theta)\partial_s\vec X^\bot\cdot\vec e_1 + \gamma'(\theta)(\partial_s\vec X\cdot\vec e_1)]\vec n \nonumber \\
&=-\gamma(\theta)(\vec n\cdot\vec e_1)\vec n + (\vec\tau\cdot\vec e_1)\gamma'(\theta)\vec n. 
\end{align}
From \eqref{eqn:model_b}, \eqref{eqn:equiv_pf4}, \eqref{eqn:equiv_pf5} and the decomposition $\vec e_1 = (\vec\tau\cdot\vec e_1)\vec\tau + (\vec n\cdot\vec e_1)\vec n$, we finally obtain \eqref{eqn:equiv}. 
\end{proof}
Selecting a test function $\varphi\in H^1(\mathbb{I})$, multiplying $r\varphi\vec n$ to \eqref{eqn:model_a}, integrating over $\Gamma(t)$, and noting  \eqref{eqn:BC3}, we have 
\begin{align}\label{eqn:varf_1}
\int\limits_{\Gamma(t)} r\partial_t\vec X\cdot\vec n\varphi ds&=\int\limits_{\Gamma(t)}-\partial_s(r\partial_s\mu)\varphi ds \nonumber \\
&=\int\limits_{\Gamma(t)}r\partial_s\mu\,\partial_s\varphi ds-\left(r\partial_s\mu\,\varphi\right)\bigg|_{s=0} ^{s=L} \nonumber \\
&=\int\limits_{\Gamma(t)}r\partial_s\mu\,\partial_s\varphi ds.
\end{align}
Then, multiplying $\vec\psi=(\psi_1,\psi_2)^\top\in\mathbb{X}$ to \eqref{eqn:equiv}, integrating it over $\bI$, using integrating by part, and thanks to the boundary conditions \eqref{eqn:BC2} and \eqref{eqn:f}, we obtain
\begin{align}\label{eqn:varf_2}
\int\limits_{\Gamma(t)}r\mu\vec n\cdot\vec \psi ds&=\int\limits_{\Gamma(t)}\partial_s\left[r\mat{B}_{q}(\theta)\partial_s\vec X\right]\cdot\vec\psi ds-\int\limits_{\Gamma(t)}\gamma(\theta)\psi_1 ds \nonumber \\
&=-\int\limits_{\Gamma(t)}\left[r\mat{B}_{q}(\theta)\partial_s\vec X\right]\cdot\partial_s\vec\psi ds + \left[r\mat{B}_{q}(\theta)\partial_s\vec X\right]\cdot\vec\psi\bigg|_{s=0} ^{s=L}-\int\limits_{\Gamma(t)}\gamma(\theta)\psi_1 ds \nonumber \\
&=-\int\limits_{\Gamma(t)}\left[r\mat{B}_{q}(\theta)\partial_s\vec X\right]\cdot\partial_s\vec\psi ds - 
\int\limits_{\Gamma(t)}\gamma(\theta)\psi_1 ds + 
r\begin{pmatrix}
\gamma(\theta)&-\gamma'(\theta) \\
\gamma'(\theta)&\gamma(\theta)
\end{pmatrix}\binom{\cos\theta}{\sin\theta}\cdot\binom{\psi_1}{\psi_2}\bigg|_{s=0} ^{s=L} \nonumber \\
&=-\int\limits_{\Gamma(t)}\left[r\mat{B}_{q}(\theta)\partial_s\vec X\right]\cdot\partial_s\vec\psi ds -\int\limits_{\Gamma(t)}\gamma(\theta)\psi_1 ds\nonumber \\
&~~~~- 
\frac{1}{\eta }\bigg[r_L\partial_t r(L,t)\psi_1(1)+r_0 \partial_t r(0,t)\psi_1(0)\bigg]+\sigma\bigg[r_L\psi_1(1)-r_0\psi_1(0)\bigg].
\end{align}

For convenience, we define $\langle\,\cdot\,\rangle$ by the $L^2$-inner product over $\bI$. 
Combining \eqref{eqn:varf_1} , \eqref{eqn:varf_2} and $ds=\partial_\rho s d\rho= | \partial_\rho\vec{X}  |d\rho$, we can obtain a new variational formulation of SSD that is different from ones in \cite{Zhao17,baojcm2022}. 
Suppose $\Gamma(0):=\vec X(\cdot, 0)\in\mathbb{X}$, to find open curves $\Gamma(t):=\vec X(\cdot,t)\in\mathbb{X}$, and $\mu(\cdot,t)\in H^1(\mathbb{I})$, such that
\begin{subequations}\label{eqn:vaf}
\begin{align}
&\left\langle r\partial_t\vec X\cdot\vec n,  
 \varphi\left | \partial_\rho\vec X \right |\right\rangle-\left\langle r\partial_\rho\mu, \partial_\rho\varphi\left | \partial_\rho\vec X \right |^{-1}\right\rangle=0, \quad\forall\varphi\in H^1(\mathbb{I}), \label{eqn:vaf_1} \\
&\left\langle r\mu\vec n, \vec \psi\left | \partial_\rho\vec X \right |\right\rangle + 
\left\langle r\mat{B}_{q}(\theta)\partial_\rho\vec X,  
 \partial_\rho\vec\psi\left | \partial_\rho\vec X \right |^{-1}\right\rangle+
\left\langle\gamma(\theta), \psi_1 \left | \partial_\rho\vec X \right |\right\rangle \nonumber \\
&~~~~+\frac{1}{\eta }\bigg[r_L\partial_t r(L, t)\psi_1(1)+r_0 \partial_t r(0,  t)\psi_1(0)\bigg]-\sigma\bigg[r_L\psi_1(1)-r_0\psi_1(0)\bigg]=0, \quad
\forall\vec\psi=(\psi_1, \psi_2)^\top\in\mathbb{X} \label{eqn:vaf_2}.
\end{align}
\end{subequations}
We can prove that the variational formulation \eqref{eqn:vf} holds the volume conservation and energy stability. 
\begin{thm}\label{lem:structure}(Volume conservation $\&$ energy stability). Assume $(\vec X(\cdot, t),\mu(\cdot, t))\in\mathbb{X}\times H^1(\mathbb{I})$ is a solution of variational formulation \label{eqn:vf}. Then there hold
\begin{equation}
\label{eqn:structure_vf}
\vol(\vec X(t))\equiv \vol(\vec X(0)), \qquad
W(t)\le W(t_1)\le W(0), \qquad t\ge t_1\ge 0,
\end{equation}
i.e., volume conservation and energy stability. 
\end{thm}

\begin{proof}
Taking the derivative of $\text{vol}(\vec X(t))$ with respect to $t$, we get 
\begin{align}
\label{eqn:structure_vf_pf1}
\frac{\mathrm{d} }{\mathrm{d} t} \text{vol}(\vec X(t))&=2\pi\frac{\mathrm{d} }{\mathrm{d} t}\int\limits_\mathbb{I}r\partial_\rho r zd\rho \nonumber \\
&=2\pi\int_\mathbb{I}(\partial_t r \partial_\rho r z + r\partial_\rho r \partial_t z)d\rho + 2\pi\int_\mathbb{I}r\partial_\rho\partial_t r\,zd\rho \nonumber \\
&=2\pi\int_\mathbb{I}(\partial_\rho r z\partial_t r + r\partial_\rho \partial_t z)d\rho - 2\pi\int_\mathbb{I}\partial_\rho(rz)\partial_t rd\rho + 2\pi(rz\partial_t r)\bigg|_{\rho=0} ^{\rho=1} \nonumber \\
&=2\pi\int_\mathbb{I}(r\partial_\rho r\partial_t z - r\partial_\rho r \partial_t r)d\rho \nonumber \\
&=2\pi\int_\mathbb{I}r\partial_t\vec X\cdot\vec n\left | \partial_\rho\vec X \right |d\rho, \qquad t\ge0.
\end{align}
Choosing $\varphi = 1$ in \eqref{eqn:vaf_1}, we obtain
\begin{equation}\label{eqn:vaf_11}
\left\langle r\partial_t\vec X\cdot\vec n, \left | \partial_\rho\vec X \right |\right\rangle=\left\langle r\partial_\rho\mu, 0 \left| \partial_\rho\vec X \right |^{-1}\right\rangle=0, \qquad t\ge0,
\end{equation}
which together with \eqref{eqn:structure_vf_pf1} imply the volume conservation law in \eqref{eqn:structure_vf}. 

Next, we take the derivative of $W(t)$ with respect to $t$ to arrive at  
\begin{align}
\frac{\mathrm{d} }{\mathrm{d} t} W(t)&=\frac{\mathrm{d} }{\mathrm{d} t}\left[2\pi\int_{0}^{L(t)} r\,\gamma(\theta)\,\left | \partial_s\vec X \right |ds + \sigma\pi(r_0 ^2-r_L ^2)\right] =\frac{\mathrm{d} }{\mathrm{d} t}\left[2\pi\int\limits_\mathbb{I} r\,\gamma(\theta)\,\left | \partial_\rho\vec X \right |d\rho - \sigma\pi(r_0 ^2-r_L ^2)\right] \nonumber \\
&=2\pi\int\limits_\mathbb{I} \partial_t r\gamma(\theta)\left | \partial_\rho\vec X \right |d\rho + 2\pi\int\limits_\mathbb{I}  r\gamma'(\theta)\partial_t \theta\left | \partial_\rho\vec X \right |d\rho + 2\pi\int\limits_\mathbb{I}  r\gamma(\theta)\partial_t\left(\left | \partial_\rho\vec X \right |\right)d\rho - 2\pi\sigma\bigg[r_0\partial_t r_0-r_L\partial_t r_L\bigg] \nonumber \\
&=2\pi\int\limits_\mathbb{I} \partial_t r\gamma(\theta)\left | \partial_\rho\vec X \right |d\rho + 2\pi\int\limits_\mathbb{I}  r\gamma'(\theta)\vec n\cdot\partial_\rho\partial_t\vec X d\rho + 2\pi\int\limits_\mathbb{I}  r\gamma(\theta)\vec\tau\cdot\partial_\rho\partial_t\vec X d\rho - 2\pi\sigma\bigg[r_0\partial_t r_0-r_L\partial_t r_L\bigg] \nonumber \\
&=2\pi\int\limits_\mathbb{I} \partial_t r\gamma(\theta)\left | \partial_\rho\vec X \right |d\rho 
+ 2\pi\int\limits_\mathbb{I}  r\bigg[\gamma'(\theta)\vec n + \gamma(\theta)\vec \tau\bigg]\cdot\partial_\rho\partial_t\vec X d\rho - 2\pi\sigma\bigg[r_0\partial_t r_0-r_L\partial_t r_L\bigg] \nonumber \\
&=2\pi\int\limits_\mathbb{I}\left[\partial_t r\gamma(\theta)\left | \partial_\rho\vec X \right | + r\mat{B}_{q}(\theta)\partial_\rho\vec X\cdot\partial_\rho\partial_t \vec X\left | \partial_\rho\vec X \right |^{-1}\right]d\rho-2\pi\sigma\bigg[r_0\partial_t r_0-r_L\partial_t r_L\bigg], \quad t\ge0.  
\end{align}
Setting $\varphi=\mu$ and $\psi=\partial_t\vec X$ in  \eqref{eqn:vaf}, we obtain 
\begin{equation}\label{eqn:vaf_22}
\frac{\mathrm{d} }{\mathrm{d} t} W(t)=-2\pi\left\langle r\partial_\rho\mu, \partial_\rho\mu\left | \partial_\rho\vec X \right |^{-1}\right\rangle-\frac{2\pi}{\eta }\bigg[r_L\left(\partial_t r(L, t)\right)^2+r_0 \left(\partial_t r(0,  t)\right)^2\bigg]\leq 0, 
\end{equation}
which shows the energy stability in \eqref{eqn:structure_vf}. 
We have completed the proof of this theorem. 
\end{proof}

\section{Structure-preserving finite element approximation}\label{sec4}
In this section, we intend to construct a structure-preserving finite element approximation for the variational formulation \eqref{eqn:vaf}, which can preserve the volume conservation and energy stability. 

\subsection{Finite element approximation}
We divide $[0,T]=\cup_{j=0}^{M-1}[t_m, t_{m+1}]$ with time steps $\ttau_m = t_{m+1}-t_m$, and the domain $\mathbb{I}$ is divided into $\mathbb{I}=\cup_{j=1}^{J}\mathbb{I}_j=\cup_{j=1}^{J}[q_{j-1}, q_{j}]$ with $q_j=jh$ and $h=J^{-1}$. Then, we define the finite element spaces
\begin{equation}
\mathbb{K}^h=\mathbb{K}^h(\mathbb{I}):=\left \{ u\in C(\mathbb{I}):u|_{\mathbb{I}_j}\in\mathbb{P}_1, \quad\forall j=1, 2, \dots, J \right \}\subseteq H^1(\mathbb{I}), \quad
\mathbb{X}^h := \mathbb{K}^h\times\mathbb{K}_0^h, \quad\mathbb{K}_0 ^h := \mathbb{K}^h\cap H_0^1(\mathbb{I}), \nonumber
\end{equation}
where $\mathbb{P}_1$ represents the space of all polynomials with degree at most $1$. 

Let $\Gamma^m(t)=\vec X^m(\cdot, t)\in\mathbb{X}$ be the approximation of $\{ \vec X(t)  \}_{t\in[0, T]}$. This gives the polygonal curves $\Gamma^m=\vec X^m(\mathbb{I})$. We assume
\begin{equation}
r^m>0\quad\text{in}~\,\rho\in(0, 1]\quad\text{and}\quad\left | \vec X^m \right | > 0\quad\text{in}~\,\rho\in(0, 1),\quad 0\le m\le M, \nonumber
\end{equation}
and introduce approximations of units tangent vector and outward normal vector 
\begin{equation}
\vec\tau ^m=\vec X_s ^m=\frac{\vec X_\rho ^m}{\left | \vec X_\rho ^m \right | } \quad\text{and}\quad\vec n^m=-(\vec\tau ^m)^\bot. \nonumber
\end{equation}
Furthermore, for any  piecewise continuous functions $\vec u$, $\vec v$, with possible jumps at notes $\{ q_j\}_{j=1}^J$, we define the mass-lumped $L^2$-inner product $\left \langle \cdot, \cdot \right \rangle^h $ 
as 
\begin{equation}
\left \langle \vec u, \vec v \right \rangle^h = \frac{1}{2}h\sum_{j=1}^{J}\left[(\vec u\cdot\vec v)(q_j^-)+(\vec u\cdot\vec v)(q_{j-1}^+)\right].
\end{equation}

We first introduce a \textbf{structure-preserving approximation} based on the variational formulation \eqref{eqn:vaf}. In this discrete scheme, the volume conservation and energy dissipation laws can be proved in theory. The discretization is given as follows. For $\Gamma(0):=\vec X(\cdot, 0)\in\mathbb{X}^h$, find $(\delta\vec X^{m+1}, \mu ^{m+1})\in\mathbb{X}^h\times\mathbb{K}^h$ with $\vec X^{m+1}=\vec X^m+\delta\vec X^{m+1}$, such that
\begin{subequations}
\label{eqn:structure}
\begin{align}
&\frac{1}{\ttau_m}\left \langle \vec X^{m+1}-\vec X^m, \varphi ^h\vec f^{m+\frac{1}{2}} \right \rangle - \left \langle r^m\partial_\rho\mu ^{m+1}, \partial_\rho\varphi^h\left | \partial_\rho\vec X ^m \right |^{-1} \right \rangle=0, \qquad\forall\varphi\in\mathbb{K}^h,  
\label{eqn:structure_a}
\\
&\left \langle \mu ^{m+1}\vec f^{m+\frac{1}{2}}, \vec\psi ^h \right \rangle + \left \langle \gamma(\theta ^{m+1}), \psi_1 ^h \left | \partial_\rho\vec X^{m+1} \right | \right \rangle + 
\left\langle r^m\mat{B}_{q}(\theta ^m)\partial_\rho\vec X^{m+1},  
 \partial_\rho\vec\psi ^h\left | \partial_\rho\vec X ^m \right |^{-1}\right\rangle \nonumber \\
 &~~~~+\frac{1}{2\eta\ttau_m }\bigg[(r_L^{m+1}+r_L^m)(r_L^{m+1}-r_L^m)\psi_1 ^h(1)+(r_0^{m+1}+r_0^m)(r_0^{m+1}-r_0^m)\psi_1 ^h(0)\bigg] \nonumber \\
 &~~~~-\frac{\sigma}{2}\bigg[(r_L^{m+1}+r_L^m)\psi_1 ^h(1)-(r_0^{m+1}+r_0^m)\psi_1 ^h(0)\bigg]=0, \quad
\forall\vec\psi=(\psi_1, \psi_2)^\top\in\mathbb{X}^h,
\label{eqn:structure_b}
\end{align}
\end{subequations}
where $\vec f^{m+\frac{1}{2}}\in [L^\infty(\bI)]^2$ represents a time-integrated approximation of $\pol f=r\,|\partial_\rho\pol X|\,\pol n$, given by
\begin{align}
\pol f^{m+\frac{1}{2}}= -\frac{1}{6}\Bigl[2r^m\,\partial_\rho\pol X^m+2r^{m+1}\,\partial_\rho\pol X^{m+1} + r^m\,\partial_\rho\pol X^{m+1} + r^{m+1}\,\partial_\rho\pol X^m\Bigr]^\perp.\label{eq:weightednormal}
\end{align}
\begin{thm}[volume conservation] Let $(\vec X^{m+1},\mu^{m+1})$ be a solution of \eqref{eqn:structure}. Then it holds that
\begin{equation}
\vol(\vec X^{m+1})-\vol(\vec X^m) = 0.\label{eq:dvolumec}
\end{equation}
\end{thm}
\begin{proof}
By taking $\varphi^h=\ttau_m$ in \eqref{eqn:structure_a} and recalling 
\begin{equation} \label{eq:dism}
\vol(\pol X^{m+1})-\vol(\pol X^m) = 2\pi\,\Bigl\langle\pol X^{m+1}-\pol X^m,~\pol f^{m+\frac{1}{2}}\Bigr\rangle,\quad\mbox{for}\quad\vec X^m\in\mathbb{X},\, \vec X^{m+1}\in\mathbb{X},\quad 0\leq m\leq M-1,
\end{equation}
it is straightforward to derive the volume conservation  \eqref{eq:dvolumec}.
\end{proof}
\begin{rem}
Due to the diverse proof process for the energy stability property of numerical formulations under different parameter values $q$, we will divide the discussion into the following two subsections.
\end{rem}

\subsection{Energy decay property: $q=0$}
\begin{thm}
Assume the matrix $\mat{B}_0(\theta)$ defined in \eqref{Matrix:Bqx} holds 
\begin{equation}\label{eq:defofk0}
\bigg[\gamma(\theta)\mat{B}_0(\theta)\,(\cos\hat{\theta}, \sin\hat{\theta})^\top\bigg]\cdot(\cos\hat{\theta}, \sin\hat{\theta})^\top\geq \gamma(\hat{\theta})^2,\quad \forall \theta,\hat{\theta}\in [-\pi, \pi],
	\end{equation}
 then it holds that
\begin{equation}
\label{eq:denergyd}
W(\vec X^{m+1})+ 2\pi\ttau_m\left \langle r^m(\partial_\rho\mu ^{m+1})^2, \left | \vec X_\rho ^m \right |^{-1} \right \rangle + \frac{\pi}{\eta\ttau_m}\left[(r_L^{m+1}+r_L^m)(r_L^{m+1}-r_L^m)^2+(r_0^{m+1}+r_0^m)(r_0^{m+1}-r_0^m)^2\right]\le W(\vec X^m). 
\end{equation}
\end{thm}
\begin{proof}
Taking $\varphi ^h=\ttau_m\mu ^{m+1}$ in \eqref{eqn:structure_a} and $\vec\psi ^h=\delta\vec X^{m+1}=(\vec X^{m+1}-\vec X^m)$ in \eqref{eqn:structure_b}, we get 
 \begin{align}\label{eqn:qz}
&\ttau_m\left \langle r^m(\partial_\rho\mu ^{m+1})^2, \left | \vec X_\rho ^m \right |^{-1} \right \rangle + \left \langle \gamma(\theta ^{m+1}), (r^{m+1}-r^m)\left | \partial_\rho\vec X^{m+1} \right | \right \rangle + 
\left\langle r^m\mat{B}_{0}(\theta ^m)\partial_\rho\vec X^{m+1},  
 \partial_\rho(\vec X^{m+1}-\vec X^m)\left | \partial_\rho\vec X ^m \right |^{-1}\right\rangle \nonumber \\
 &~~~~+\frac{1}{2\eta\ttau_m }\bigg[(r_L^{m+1}+r_L^m)(r_L^{m+1}-r_L^m)^2+(r_0^{m+1}+r_0^m)(r_0^{m+1}-r_0^m)^2\bigg] \nonumber \\
 &~~~~-\frac{\sigma}{2}\bigg[(r_L^{m+1}+r_L^m)(r_L^{m+1}-r_L^m)-(r_0^{m+1}+r_0^m)(r_0^{m+1}-r_0^m)\bigg]=0.  
 \end{align}
Due to the symmetric positive definite of the matrix $\mat{B}_0(\theta)$, we have 
\begin{equation}\label{eq:inequa}
\mat{B}_0(\theta)\vec v\cdot(\vec v-\vec w)\ge\frac{1}{2}\mat{B}_0(\theta)\vec v\cdot\vec v - \frac{1}{2}\mat{B}_0(\theta)\vec w\cdot\vec w, \qquad\forall\vec v, \vec w\in\mathbb{R}^2. 
\end{equation}
From \eqref{eq:defofk0}, we can obtain
\begin{equation}
\gamma(\theta ^m)\mat{B}_0(\theta ^m)\partial_\rho\vec X^{m+1}\cdot\partial_\rho\vec X^{m+1}\ge\gamma^2(\theta ^{m+1})\left | \vec X_\rho ^{m+1} \right |^2. 
\end{equation}
Then we have 
\begin{align}\label{eqn:zm}
&\left\langle r^m\mat{B}_{0}(\theta ^m)\partial_\rho\vec X^{m+1},  
 \partial_\rho(\vec X^{m+1}-\vec X^m)\left | \partial_\rho\vec X ^m \right |^{-1}\right\rangle + \frac{1}{2\eta\ttau_m }\bigg[(r_L^{m+1}+r_L^m)(r_L^{m+1}-r_L^m)^2+(r_0^{m+1}+r_0^m)(r_0^{m+1}-r_0^m)^2\bigg] \nonumber \\
 &~~~~-\frac{\sigma}{2}\bigg[(r_L^{m+1}+r_L^m)(r_L^{m+1}-r_L^m)-(r_0^{m+1}+r_0^m)(r_0^{m+1}-r_0^m)\bigg] \nonumber \\
 &~~~~\ge \frac{1}{2} \left\langle r^m\mat{B}_{0}(\theta ^m)\partial_\rho\vec X^{m+1}, 
 \partial_\rho\vec X^{m+1}\left | \partial_\rho\vec X ^m \right |^{-1}\right\rangle - \frac{1}{2} \left\langle r^m\mat{B}_{0}(\theta ^m)\partial_\rho\vec X^m,  
 \partial_\rho\vec X^m\left | \partial_\rho\vec X ^m \right |^{-1}\right\rangle \nonumber \\
 &~~~~+\frac{1}{2\eta\ttau_m }\bigg[(r_L^{m+1}+r_L^m)(r_L^{m+1}-r_L^m)^2+(r_0^{m+1}+r_0^m)(r_0^{m+1}-r_0^m)^2\bigg] + \frac{\sigma}{2}\bigg[(r_0^{m+1})^2-(r_0^m)^2-(r_L^{m+1})^2+(r_L^m)^2\bigg] \nonumber \\
 &~~~~=\frac{1}{2} \left\langle r^m, \mat{B}_{0}(\theta ^m)\partial_\rho\vec X^{m+1}\cdot 
 \partial_\rho\vec X^{m+1}\left | \partial_\rho\vec X ^m \right |^{-1} + \gamma(\theta ^m)\left | \partial_\rho\vec X ^m \right |\right\rangle - \left\langle r^m, \gamma(\theta ^m)\left | \partial_\rho\vec X ^m \right |\right\rangle \nonumber \\
 &~~~~+\frac{1}{2\eta\ttau_m }\bigg[(r_L^{m+1}+r_L^m)(r_L^{m+1}-r_L^m)^2+(r_0^{m+1}+r_0^m)(r_0^{m+1}-r_0^m)^2\bigg] + \frac{\sigma}{2}\bigg[(r_0^{m+1})^2-(r_0^m)^2-(r_L^{m+1})^2 + (r_L^m)^2\bigg] \nonumber \\
 &~~~~\ge \left\langle r^m, \sqrt{\gamma(\theta ^m)(\mat{B}_{0}(\theta ^m)\partial_\rho\vec X^{m+1})\cdot 
 \partial_\rho\vec X^{m+1}}\right\rangle - \left\langle r^m, \gamma(\theta ^m)\left | \partial_\rho\vec X ^m \right |\right\rangle \nonumber \\
 &~~~~+\frac{1}{2\eta\ttau_m }\bigg[(r_L^{m+1}+r_L^m)(r_L^{m+1}-r_L^m)^2+(r_0^{m+1}+r_0^m)(r_0^{m+1}-r_0^m)^2\bigg] + \frac{\sigma}{2}\bigg[(r_0^{m+1})^2-(r_0^m)^2-(r_L^{m+1})^2 + (r_L^m)^2\bigg] \nonumber \\
 &~~~~\ge \left\langle r^m, \gamma(\theta ^{m+1})\left | \partial_\rho\vec X ^{m+1} \right |\right\rangle - \left\langle r^m, \gamma(\theta ^m)\left | \partial_\rho\vec X ^m \right |\right\rangle \nonumber \\
 &~~~~+\frac{1}{2\eta\ttau_m }\bigg[(r_L^{m+1}+r_L^m)(r_L^{m+1}-r_L^m)^2+(r_0^{m+1}+r_0^m)(r_0^{m+1}-r_0^m)^2\bigg] + \frac{\sigma}{2}\bigg[(r_0^{m+1})^2-(r_0^m)^2-(r_L^{m+1})^2 + (r_L^m)^2\bigg].
\end{align}
Combining \eqref{eqn:qz} with \eqref{eqn:zm}, we derive the energy stability \eqref{eq:denergyd}. 
\end{proof}
\begin{rem}
We define minimal stabilizing function $\mathscr S_0(\theta)$ as 
\begin{equation}
\mathscr S_0(\theta):=inf\left \{ \mathscr S(\theta)|\bigg[\gamma(\theta)\mat{B}_0(\theta)\,(\cos\hat{\theta}, \sin\hat{\theta})^\top\bigg]\cdot(\cos\hat{\theta}, \sin\hat{\theta})^\top\geq \gamma(\hat{\theta})^2, \forall\hat{\theta}\in[-\pi, \pi] \right \}, \quad\theta\in[-\pi, \pi]. 
\end{equation}
The inequality \eqref{eq:inequa} can be satisfied if $\gamma(\theta)=\gamma(\pi+\theta)$ and $\mathscr S(\theta)\ge \mathscr S_0(\theta)$ for $\theta\in[-\pi, \pi]$ (see \cite{bao2023symmetrized}). 
\end{rem}

\subsection{Energy decay property: $q=1$}
\begin{thm}
Assume 
the matrix $\mat{B}_1(\theta)$ defined in \eqref{Matrix:Bqx} satisfies 
\begin{equation}\label{eq:defofk1}
\frac{1}{\left | \vec v \right |}\left(\mat{B}_1(\theta)\vec w\right)\cdot(\vec w-\vec v)\ge\left | \vec w \right |\gamma(\hat{\theta})-\left | \vec v \right |\gamma(\theta), 
	\end{equation}
 then it holds that
\begin{equation}
\label{eq:denergyd2}
W(\vec X^{m+1})+ 2\pi\ttau_m\left \langle r^m(\partial_\rho\mu ^{m+1})^2, \left | \vec X_\rho ^m \right |^{-1} \right \rangle + \frac{\pi}{\eta\ttau_m}\left[(r_L^{m+1}+r_L^m)(r_L^{m+1}-r_L^m)^2+(r_0^{m+1}+r_0^m)(r_0^{m+1}-r_0^m)^2\right]\le W(\vec X^m). 
\end{equation}
\end{thm}
\begin{proof}
Taking $\varphi ^h=\ttau_m\mu ^{m+1}$ in \eqref{eqn:structure_a} and $\vec\psi ^h=\delta\vec X^{m+1}=(\vec X^{m+1}-\vec X^m)$ in \eqref{eqn:structure_b}, we get 
 \begin{align}\label{eqn:qz2}
&\ttau_m\left \langle r^m(\partial_\rho\mu ^{m+1})^2, \left | \vec X_\rho ^m \right |^{-1} \right \rangle + \left \langle \gamma(\theta ^{m+1}), (r^{m+1}-r^m)\left | \partial_\rho\vec X^{m+1} \right | \right \rangle + 
\left\langle r^m\mat{B}_{1}(\theta ^m)\partial_\rho\vec X^{m+1},  
 \partial_\rho(\vec X^{m+1}-\vec X^m)\left | \partial_\rho\vec X ^m \right |^{-1}\right\rangle \nonumber \\
 &~~~~+\frac{1}{2\eta\ttau_m }\bigg[(r_L^{m+1}+r_L^m)(r_L^{m+1}-r_L^m)^2+(r_0^{m+1}+r_0^m)(r_0^{m+1}-r_0^m)^2\bigg] \nonumber \\
 &~~~~-\frac{\sigma}{2}\bigg[(r_L^{m+1}+r_L^m)(r_L^{m+1}-r_L^m)-(r_0^{m+1}+r_0^m)(r_0^{m+1}-r_0^m)\bigg]=0.  
 \end{align}
From \eqref{eq:defofk1}, we can obtain
\begin{equation}
\frac{1}{\left | \partial_\rho\vec X^m  \right |}\left(\mat{B}_1(\theta ^m)\partial_\rho\vec X^{m+1}\right)\cdot(\partial_\rho\vec X^{m+1}-\partial_\rho\vec X^m)\ge\left | \partial_\rho\vec X^{m+1} \right |\gamma(\theta ^{m+1})-\left | \partial_\rho\vec X^m \right |\gamma(\theta ^m).  
\end{equation}
Then we have 
\begin{align}\label{eqn:zm2}
&\left\langle r^m\mat{B}_{1}(\theta ^m)\partial_\rho\vec X^{m+1},  
 \partial_\rho(\vec X^{m+1}-\vec X^m)\left | \partial_\rho\vec X ^m \right |^{-1}\right\rangle + \frac{1}{2\eta\ttau_m }\bigg[(r_L^{m+1}+r_L^m)(r_L^{m+1}-r_L^m)^2+(r_0^{m+1}+r_0^m)(r_0^{m+1}-r_0^m)^2\bigg] \nonumber \\
 &~~~~-\frac{\sigma}{2}\bigg[(r_L^{m+1}+r_L^m)(r_L^{m+1}-r_L^m)-(r_0^{m+1}+r_0^m)(r_0^{m+1}-r_0^m)\bigg] \nonumber \\
 &~~~~\ge \left\langle r^m, \gamma(\theta ^{m+1})\left | \partial_\rho\vec X ^{m+1} \right |\right\rangle - \left\langle r^m, \gamma(\theta ^m)\left | \partial_\rho\vec X ^m \right |\right\rangle \nonumber \\
 &~~~~+\frac{1}{2\eta\ttau_m }\bigg[(r_L^{m+1}+r_L^m)(r_L^{m+1}-r_L^m)^2+(r_0^{m+1}+r_0^m)(r_0^{m+1}-r_0^m)^2\bigg] + \frac{\sigma}{2}\bigg[(r_0^{m+1})^2-(r_0^m)^2-(r_L^{m+1})^2 + (r_L^m)^2\bigg].
 \end{align}
 Combining \eqref{eqn:qz2} with \eqref{eqn:zm2}, we derive the energy stability  \eqref{eq:denergyd2}.
\end{proof}

\begin{rem}
Introducing auxiliary functions are defined by $P_\alpha(\theta, \hat{\theta}), Q(\theta, \hat{\theta})$ 
\begin{subequations}
\begin{align}
&P_\alpha(\theta, \hat{\theta}):=2\sqrt{(\gamma(\theta)+\alpha(-\sin\hat{\theta}\cos\theta+\cos\hat{\theta}\sin\theta)^2)\gamma(\theta)}, \quad\forall\theta, \hat{\theta}\in[-\pi, \pi], \quad\alpha\ge 0, \\
&Q(\theta, \hat{\theta}):=\gamma(\hat{\theta})+\gamma(\theta)(\sin\theta\sin\hat{\theta}+\cos\theta\cos\hat{\theta})+\gamma'(\theta)(-\sin\hat{\theta}\cos\theta+\cos\hat{\theta}\sin\theta), \quad\forall\theta, \hat{\theta}\in[-\pi, \pi], 
\end{align}
\end{subequations}
then we define the minimal stabilizing function $\mathscr S_0(\theta)$ as 
\begin{equation}
\mathscr S_0(\theta):=inf\left \{ \alpha\ge 0:P_\alpha(\theta, \hat{\theta})-Q(\theta, \hat{\theta})\ge 0, \forall\hat{\theta}\in[-\pi, \pi] \right \}, \quad\theta\in[-\pi, \pi]. 
\end{equation}
The inequality \eqref{eq:defofk1} can be satisfied if $(-\sin\theta, \cos\theta)^\top=-\frac{\vec v}{\left | \vec v \right |}$, $(-\sin\hat{\theta}, \cos\hat{\theta})^\top=-\frac{\vec w}{\left | \vec w \right |}$ are nonzero vectors, $\gamma(\theta)$ satisfies 
$3\gamma(\theta)>\gamma(\pi + \theta)$ and stabilizing function satisfies $\mathscr S(\theta)\ge\mathscr S_0(\theta)$ for $\theta\in[-\pi, \pi]$ (see \cite{bao2024structure}). 
\end{rem}

\section{Approximations with improved mesh quality}\label{sec5}

Reviewing \eqref{eqn:vaf_2}, we propose a new variational formulation, which can improve the mesh quality in the context of discretization. To this end, we introduce
\begin{equation}\label{eqn:lam}
\mu=\kappa-\lambda\quad\text{with}\quad\lambda=\frac{(\gamma(\theta)\vec n-\gamma'(\theta)\vec \tau)\cdot\vec e_1}{r}. 
\end{equation}
Given an initial open curve $\Gamma(0):=\vec X(\cdot, 0)\in\mathbb{X}$, another variational formulation is to find open curves $\Gamma(t):=\vec X(\cdot,t)\in\mathbb{X}$, and $\mu(\cdot,t)\in H^1(\mathbb{I})$, such that
\begin{subequations}\label{eqn:vafo}
\begin{align}
&\left\langle r\partial_t\vec X\cdot\vec n,  
 \varphi\left | \partial_\rho\vec X \right |\right\rangle-\left\langle r\partial_\rho\mu, \partial_\rho\varphi\left | \partial_\rho\vec X \right |^{-1}\right\rangle=0, \quad\forall\varphi\in H^1(\mathbb{I}), \label{eqn:vafo_1} \\
&\left\langle\mu+\lambda, \vec n\cdot\vec \psi\left | \partial_\rho\vec X \right |\right\rangle + 
\left\langle \mat{B}_{q}(\theta)\partial_\rho\vec X,  
 \partial_\rho\vec\psi\left | \partial_\rho\vec X \right |^{-1}\right\rangle \nonumber \\
&~~~~+\frac{1}{\eta }\bigg[\partial_t r(L, t)\psi_1(1)+ \partial_t r(0,  t)\psi_1(0)\bigg]-\sigma\bigg[\psi_1(1)-\psi_1(0)\bigg]=0, \quad
\forall\vec\psi=(\psi_1, \psi_2)^\top\in\mathbb{X} \label{eqn:vafo_2}.
\end{align}
\end{subequations}
Similarly, we can also obtain the volume conservation and energy decay laws for the variational formulation \eqref{eqn:vafo}.

Then, we introduce a linear discretized scheme based on variational formulation \eqref{eqn:vafo_1}. Similar as \cite{Barrett19,li2024parametric}, for  $\forall\rho_0=0$, there holds  
\begin{align}
\lim_{\rho \to \rho _0}\lambda(\rho, t)&=\lim_{\rho \to \rho _0}\frac{(\gamma(\theta)\vec n-\gamma'(\theta)\vec \tau)\cdot \vec e_1}{r}=\lim_{\rho \to \rho _0}\frac{\partial_\rho\left(\gamma(\theta)\vec n-\gamma'(\theta)\vec \tau\right)\cdot \vec e_1}{\partial_\rho r} \nonumber \\
&=\partial_s\left(\gamma(\theta)\vec n-\gamma'(\theta)\vec \tau\right)|_{\rho=\rho _0}\cdot\vec\tau(\rho _0, t)=-\kappa(\rho _0, t),\qquad t\in[0,T]. \nonumber 
\end{align}
To avoid the degeneracy in the discretization on $\rho=0$, on recalling \eqref{eqn:lam} we define 
\begin{equation}
[\lambda ^{m+\frac{1}{2}}(\mu ^{m+1})](q_j)=\left\{\begin{matrix}
 -\frac{1}{2}\mu ^{m+1}(q_j) ,& q_j=0,\\
 &\\
 \frac{\vec\omega ^{m}(q_j)\cdot \vec e_1}{\vec X^m(q_j)\cdot\vec e_1}, & \text{otherwise}, 
\end{matrix}\right.\qquad \text{for}\quad \vec\omega ^m\in[\mathbb{K}^h]^2.
\end{equation}
Then we can obtain the \textbf{linear approximation}. For $\Gamma(0):=\vec X(\cdot, 0)\in\mathbb{X}^h$, find $(\delta\vec X^{m+1}, \mu ^{m+1})\in\mathbb{X}^h\times\mathbb{K}^h$ with $\vec X^{m+1}=\vec X^m+\delta\vec X^{m+1}$, such that
\begin{subequations}
\label{eqn:structurexian}
\begin{align}
&\frac{1}{\ttau_m}\left \langle r^m(\vec X^{m+1}-\vec X^m), \varphi ^h\vec n^m\left | \partial_\rho\vec X ^m \right | \right \rangle - \left \langle r^m\partial_\rho\mu ^{m+1}, \partial_\rho\varphi^h\left | \partial_\rho\vec X ^m \right |^{-1} \right \rangle=0, \qquad\forall\varphi\in\mathbb{K}^h, 
\label{eqn:structure_axian}
\\
&\left \langle \mu ^{m+1}+\lambda ^{m+\frac{1}{2}}, \vec n^m\cdot\vec\psi ^h \left | \partial_\rho\vec X^m \right | \right \rangle + 
\left\langle \mat{B}_{q}(\theta ^m)\partial_\rho\vec X^{m+1},  
 \partial_\rho\vec\psi ^h\left | \partial_\rho\vec X ^m \right |^{-1}\right\rangle \nonumber \\
 &~~~~+\frac{1}{\eta\ttau }\bigg[(r^{m+1}_L-r^m_L)\psi_1 ^h(1)+ (r^{m+1}_0-r^m_0)\psi_1 ^h(0)\bigg]-\sigma\bigg[\psi_1 ^h(1)-\psi_1 ^h(0)\bigg]=0, \quad
\forall\vec\psi=(\psi_1, \psi_2)^\top\in\mathbb{X}^h. 
\label{eqn:structure_bxian}
\end{align}
\end{subequations}
The scheme \eqref{eqn:structurexian} it can improve mesh quality efficiently. But properties of volume conservation and energy stability properties cannot be theoretically proofed. We next consider a \textbf{volume-preserving approximation}. For $\Gamma(0):=\vec X(\cdot, 0)\in\mathbb{X}^h$, find $(\delta\vec X^{m+1}, \mu ^{m+1})\in\mathbb{X}^h\times\mathbb{K}^h$ with $\vec X^{m+1}=\vec X^m+\delta\vec X^{m+1}$, such that
\begin{subequations}
\label{eqn:structurefxian}
\begin{align}
&\frac{1}{\ttau_m}\left \langle \vec X^{m+1}-\vec X^m, \vec f^{m+\frac{1}{2}}\varphi^h\right \rangle - \left \langle r^m\partial_\rho\mu ^{m+1}, \partial_\rho\varphi^h\left | \partial_\rho\vec X ^m \right |^{-1} \right \rangle=0, \qquad\forall\varphi\in\mathbb{K}^h,
\label{eqn:structure_afxian}
\\
&\left \langle \mu ^{m+1}+\lambda ^{m+\frac{1}{2}}, \vec n^m\cdot\vec\psi ^h \left | \partial_\rho\vec X^m \right | \right \rangle + 
\left\langle \mat{B}_{q}(\theta ^m)\partial_\rho\vec X^{m+1},  
 \partial_\rho\vec\psi ^h\left | \partial_\rho\vec X ^m \right |^{-1}\right\rangle \nonumber \\
 &~~~~+\frac{1}{\eta\ttau }\bigg[(r^{m+1}_L-r^m_L)\psi_1 ^h(1)+ (r^{m+1}_0-r^m_0)\psi_1 ^h(0)\bigg]-\sigma\bigg[\psi_1 ^h(1)-\psi_1 ^h(0)\bigg]=0, \quad
\forall\vec\psi=(\psi_1, \psi_2)^\top\in\mathbb{X}^h. 
\label{eqn:structure_bfxian}
\end{align}
\end{subequations}
For the scheme \eqref{eqn:structurefxian}, volume conservation can be satisfied by choosing $\varphi^h=\ttau_m$ in \eqref{eqn:structure_afxian}. Although, as the scheme \eqref{eqn:structurexian} energy stability property cannot be proved in theory, the mesh quality remains nice. 

\section{Numerical results}\label{sec6}

In this section, we will present some experimental test numerical schemes and simulate SSD with axisymmetric geometry. We denote the schemes \eqref{eqn:structurexian}, \eqref{eqn:structurefxian} and \eqref{eqn:structure} as $\mathbf{L}$-method, $\mathbf{V}$-method and $\mathbf{P}$-method for brevity. We employ uniform time step size with $\ttau_m=\ttau=\frac{T}{M}$ for $m=0,\dots,M-1$. In order to better observe the effects of these methods, we introduce the volume loss function 
\begin{equation}
\left.\Delta V(t)\right|_{t=t_m}= \frac{\vol(\vec X^m)-\vol(\vec X^0)}{\vol(\vec X^0)},\qquad m\geq 0,\nn
\end{equation}
where $\vol(\vec X^m)$ is denoted by 
\begin{equation}
\vol(\vec X^m)=\pi\left \langle (\vec X^m \cdot\vec e_1)^2\vec n^m, \vec e_1\left | \partial_\rho\vec X^m \right | \right \rangle. \nonumber
\end{equation}
We test convergence by quantifying the difference between axisymmetric surfaces enclosed by curves $\Gamma_1$ and $\Gamma_2$. Therefore, we adopt the manifold distance in as 
\begin{equation}
\text{Md}(\Gamma_1, \Gamma_2):=\left |(\Omega_1\backslash\Omega_2)\cup(\Omega_2\backslash\Omega_1) \right |=\left |\Omega_1 \right |+\left |\Omega_2 \right |-2\left |\Omega_1\cap\Omega_2 \right |, \nonumber
\end{equation}
where $\Omega_i$ represents the region enclosed by $\Gamma_i$, and $| \cdot |$ denotes the area of region. Let $\vec X^m$ denote numerical approximation of surface with mesh size $h$ and time step $\ttau$, then introduce approximate solution between interval $[t_m, t_{m+1}]$ as
\begin{equation}
\vec X_{h, \ttau}(\rho, t)=\frac{t-t_m}{\ttau}\vec X^m(\rho)+\frac{t_m-t}{\ttau}\vec X^{m+1}(\rho), \quad\rho\in\mathbb{I}. 
\end{equation}
Then we define the errors by  
\begin{equation}
e_{h, \ttau}(t)=\text{Md}(\Gamma_{h, \ttau}, \Gamma_{\frac{h}{2}, \frac{\ttau}{4}}), \quad\tilde{e}_{h, \ttau}(t)=\text{Md}(\Gamma_{h, \ttau}, \Gamma_{\frac{h}{2}, \ttau}). 
\end{equation}

\textbf{Example 1}:
We test the errors and convergence rate of $\mathbf{P}$-method with respect to different types of anisotropy function $\gamma(\theta)$, including the two cases in this example:
\begin{itemize}
    \item 4-fold anisotropy: $\gamma(\theta)=1+\beta\cos(4\theta)$;
    \item 3-fold anisotropy: $\gamma(\theta)=1+\beta\cos(3\theta)$.
\end{itemize}
For the 4-fold anisotropy, we adopt $\mathbf{P}$-method with 
the surface energy matrix $\mat{B}_0(\theta)$, while for the 3-fold anisotropy, we adopt $\mathbf{P}$-method with $\mat{B}_1(\theta)$.
We choose the semi-ellipse rotation as the initial shape, with the major axis $4$ and minor axis $2$. 
The numerical errors and orders of the $\mathbf{P}$-method with 4-fold and 3-fold anisotropies are shown in Table \ref{Table:1} and Table \ref{Table:2}, respectively. 
We can observe that the $\mathbf{P}$-method has a good numerical approximation for both cases. Furthermore, we plot the volume loss and energy ratio $E(t)/E(0)$ of $\mathbf{P}$-method in  Figures \ref{fig:1}-\ref{fig:2}, which are accordant with our theoretical analysis.   
\begin{table}[!htp]
\centering
\def\temptablewidth{0.75\textwidth}
\vspace{-10pt}
\caption{Errors and convergence rate of the numerical solution for the interface using the $\mathbf{P}$-method with respect to matrix $\mat{B}_0(\theta)$, where $\beta=0, 0.05, 0.07$, and $\ttau_0=3/5$, $h_0=1/10$.}
{\rule{\temptablewidth}{1pt}}
\begin{tabular*}{\temptablewidth}{@{\extracolsep{\fill}}c|cc|cc|cc}
 & \multicolumn{2}{c|}{$\beta=0$} & \multicolumn{2}{c|}{$\beta=0.05$} &\multicolumn{2}{c}{$\beta=0.07$} \\ \hline
$(h,\ \ttau)$  
&$e_{h,\ttau} $ & order &$e_{h,\ttau}$ & order &$e_{h,\ttau} $ & order  \\ \hline
$(h_0, \ttau_0)$  
& 1.1881E-1 & -    & 1.5001E-1 &-&   2.0507E-1 &- \\ \hline
$(\frac{h_0}{2}, \frac{\ttau_0}{2^2})$ 
& 1.9395E-2 & 2.6149 & 2.3318E-2 & 2.6855 &3.2324E-2 & 2.6655 \\ \hline
$(\frac{h_0}{2^2}, \frac{\ttau_0}{2^4})$ 
& 3.7954E-3 & 2.3533 & 4.7194E-3 & 2.3048 & 8.3649E-3 &1.9502 
 \\ \hline
$(\frac{h_0}{2^3}, \frac{\ttau_0}{2^6})$ 
& 8.5268E-4 & 2.1542 & 1.0602E-3 & 2.1542 & 1.9785E-3 &2.0799 
 \end{tabular*}
 {\rule{\temptablewidth}{1pt}}
\label{Table:1}
\end{table}

\begin{table}[!htp]
\centering
\def\temptablewidth{0.75\textwidth}
\vspace{-10pt}
\caption{Errors and convergence rate of the numerical solution for the interface using the $\mathbf{P}$-method with respect to matrix $\mat{B}_1(\theta)$, where $\beta=0, 0.05, 0.2$, and $\ttau_0=3/5$, $h_0=1/10$.}
{\rule{\temptablewidth}{1pt}}
\begin{tabular*}{\temptablewidth}{@{\extracolsep{\fill}}c|cc|cc|cc}
 & \multicolumn{2}{c|}{$\beta=0$} & \multicolumn{2}{c|}{$\beta=0.05$} &\multicolumn{2}{c}{$\beta=0.2$} \\ \hline
$(h,\ \ttau)$  
&$e_{h,\ttau} $ & order &$e_{h,\ttau}$ & order &$e_{h,\ttau} $ & order  \\ \hline
$(h_0, \ttau_0)$  
& 1.7833E-1 & -    & 1.6144E-1 &-&   1.7970E-1 &- \\ \hline
$(\frac{h_0}{2}, \frac{\ttau_0}{2^2})$ 
& 3.6075E-2 & 2.3055 & 3.8652E-2 & 2.0624 &5.7118E-2 & 1.6536 \\ \hline
$(\frac{h_0}{2^2}, \frac{\ttau_0}{2^4})$ 
& 6.1600E-3 & 2.5500 & 7.0908E-3 & 2.4465 & 9.8366E-3 &2.5377 
 \\ \hline
$(\frac{h_0}{2^3}, \frac{\ttau_0}{2^6})$ 
& 1.4568E-4 & 2.0801 & 1.5961E-3 & 2.1514 & 2.0636E-3 &2.2530 
 \end{tabular*}
 {\rule{\temptablewidth}{1pt}}
\label{Table:2}
\end{table}

\begin{figure}[!htp]
\centering
\includegraphics[width=0.45\textwidth]{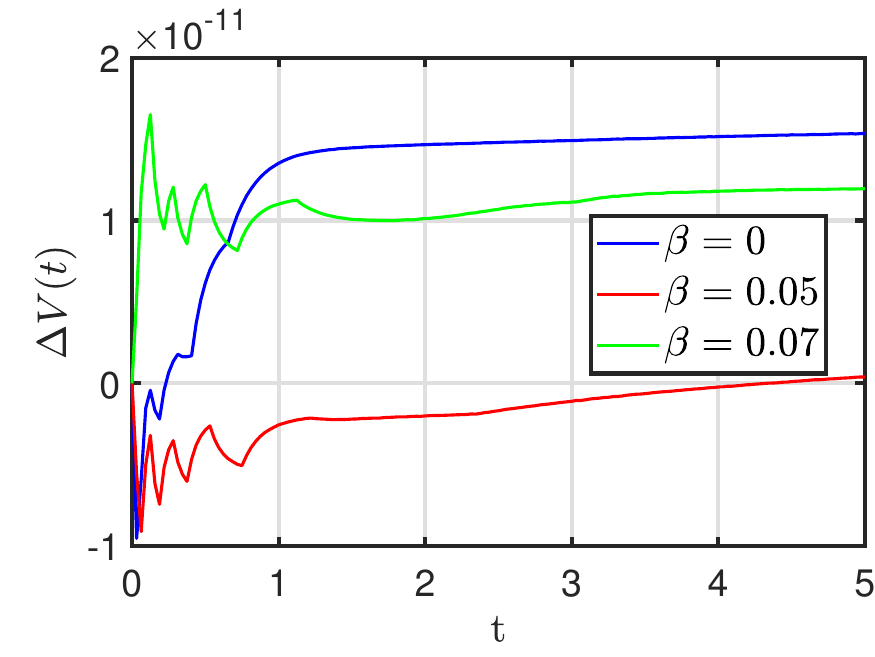}\hspace{0.5cm}
\includegraphics[width=0.45\textwidth]{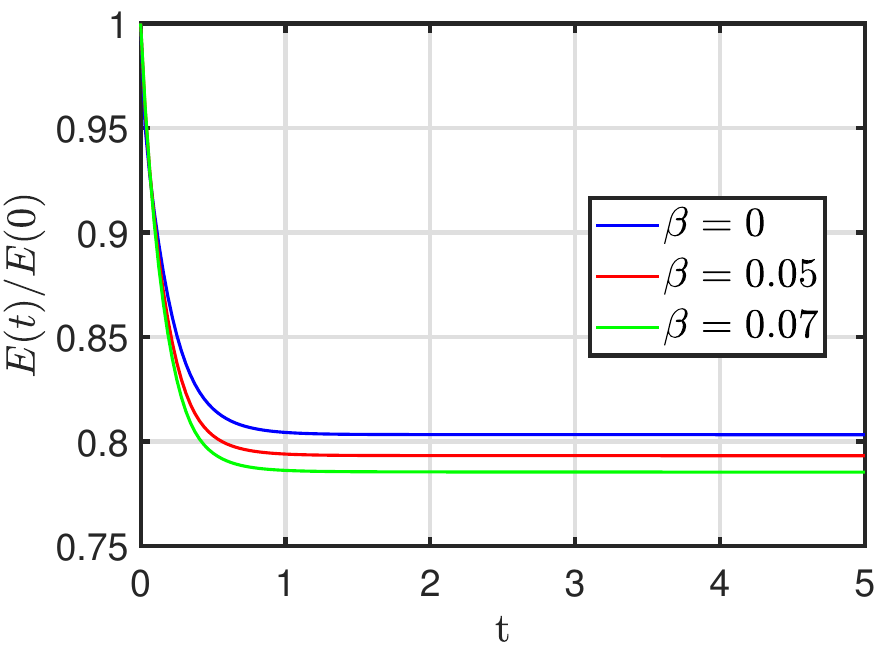}
\caption{The time history of  the relative volume loss $\Delta V(t)$ and the energy $E(t)/E(0)$ using the $\mathbf{P}$-method with respect to matrix $\mat{B}_0(\theta)$, where $h=1/80$, $\ttau=1/160$, and $\beta=0, 0.05, 0.07$. }
\label{fig:1}
\end{figure}

\begin{figure}[!htp]
\centering
\includegraphics[width=0.45\textwidth]{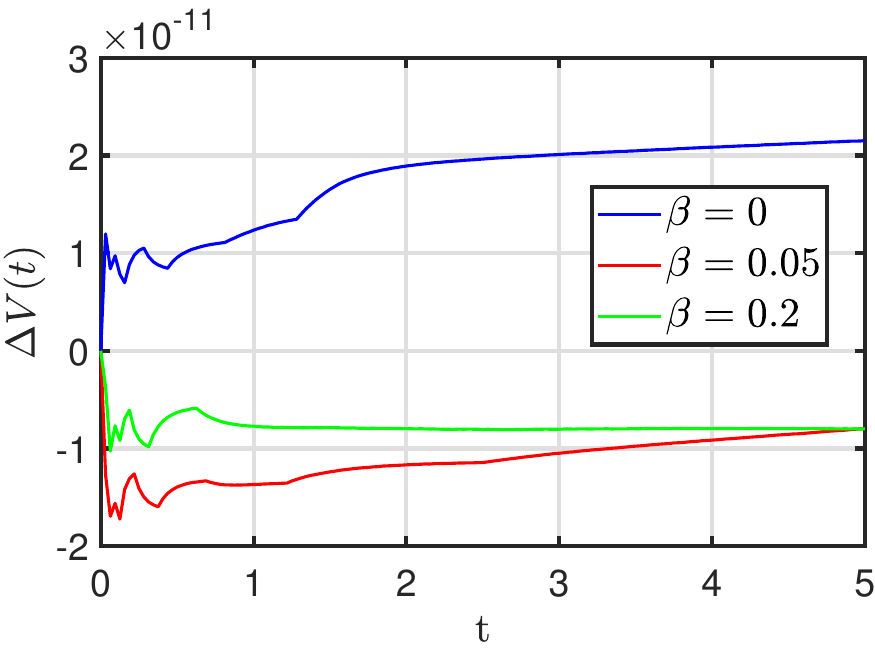}\hspace{0.5cm}
\includegraphics[width=0.45\textwidth]{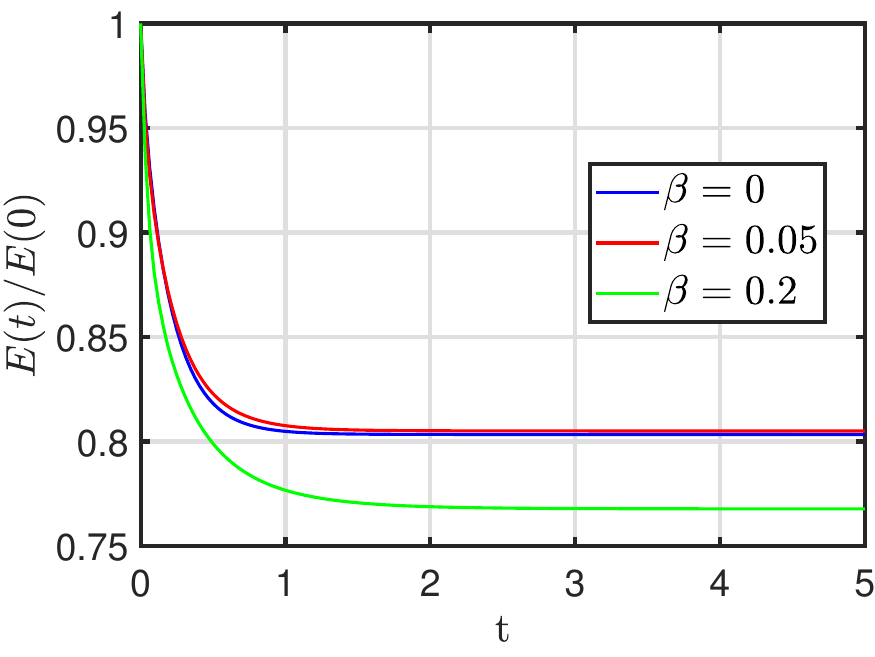}
\caption{The time history of  the relative volume loss $\Delta V(t)$  and the energy $E(t)/E(0)$ using the $\mathbf{P}$-method with respect to matrix $\mat{B}_1(\theta)$, where $h=1/80$, $\ttau=1/160$, and $\beta=0, 0.05, 0.2$. }
\label{fig:2}
\end{figure}

\textbf{Example 2}:
In this example, we compare the mesh quality of the three numerical schemes with two types of anisotropic surface energy matrices. To this end, we define the mesh ratio at $t_m$ by 
\begin{align*}
    R^h(t_m) := \frac{\max_{1\leq j\leq N} |\vec X_j^m-\vec X_{j-1}^m|}{\min_{1\leq j\leq N}|\vec X_j^m-\vec X_{j-1}^m|},\qquad m>0.
\end{align*}
We choose the same initial value as Example 1. 
Figure \ref{fig:exa2_fig1} depicts time evolution of the mesh ratio for the $\mathbf{L}$-method, $\mathbf{V}$-method and $\mathbf{P}$-method, with 4-fold anisotropy: $\gamma(\theta)=1+\beta\cos(4\theta)$. It can be clearly observed from Figure \ref{fig:exa2_fig1} that
\begin{itemize}
    \item for the weakly anisotropic case with the parameter $\beta=0.05$,  
    $\mathbf{L}$-method, $\mathbf{V}$-method and $\mathbf{P}$-method maintain good meshes, as the three mesh ratio curves approach to a same constant $C\approx 2.4$;
    \item for the strongly anisotropic case with the parameter $\beta=0.3$, 
    the mesh ratio curves of the $\mathbf{L}$-method and $\mathbf{V}$-method approach to a same constant $C\approx 20$, and the mesh ratio curve of the $\mathbf{P}$-method approaches to a relative bigger constant $C\approx 60$.
\end{itemize}
The tests above indicate that the mesh quality remains largely consistent for weakly anisotropic cases across the $\mathbf{L}$-method, $\mathbf{V}$-method, and $\mathbf{P}$-method. 
However, for strongly anisotropic cases, the mesh quality of the $\mathbf{L}$-method and $\mathbf{V}$-method is superior to that of the $\mathbf{V}$-method. However, even in strongly anisotropic cases, we can also affirm that the $\mathbf{V}$-method still maintains relatively good mesh quality. The similar tests are also given for the 3-fold anisotropy: $\gamma(\theta)=1+\beta\cos(3\theta)$, see Figure \ref{fig:exa2_fig2}.

\begin{figure}[!htp]
\centering
\includegraphics[width=0.45\textwidth]{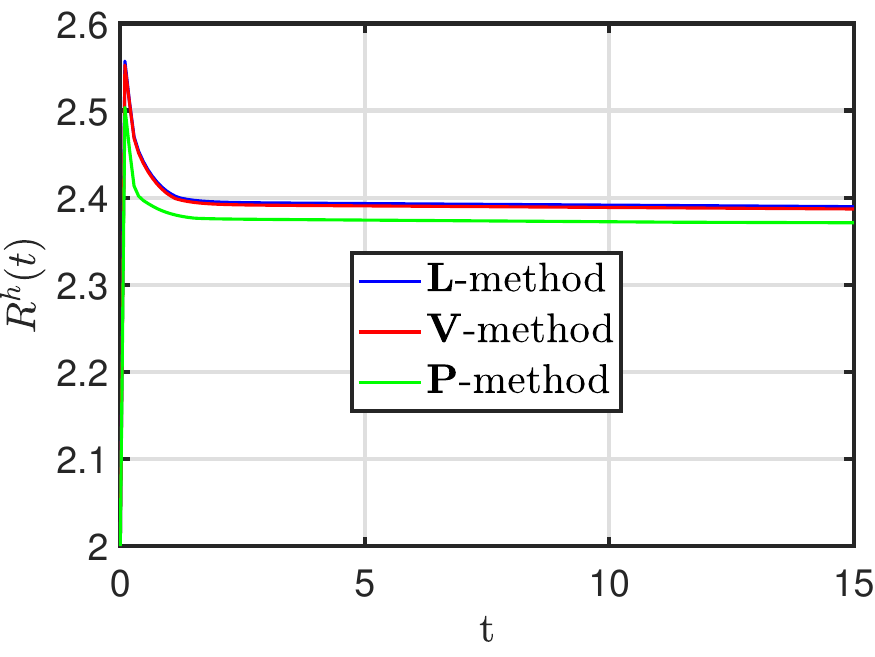}\hspace{0.5cm}
\includegraphics[width=0.45\textwidth]{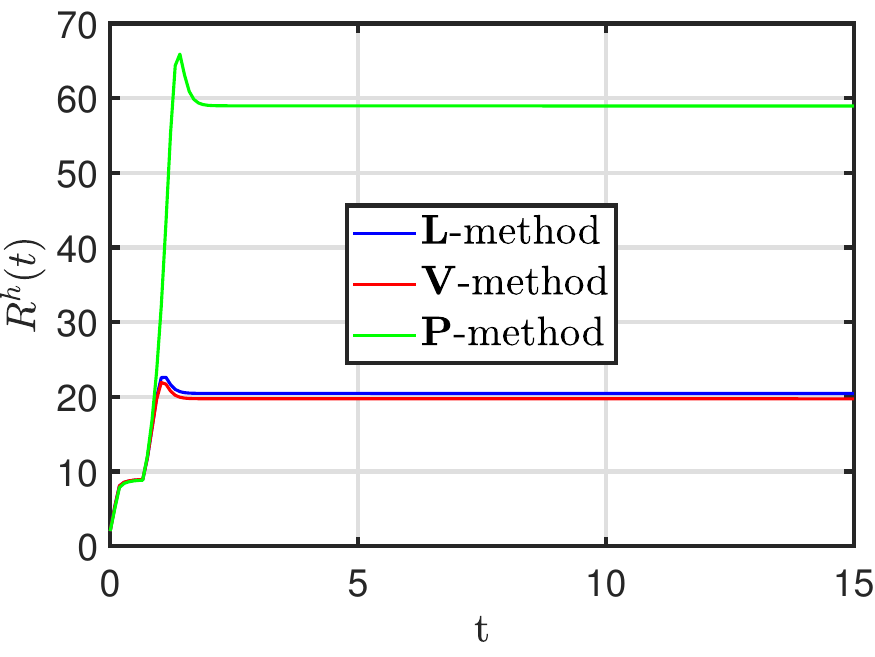}\hspace{0.5cm}
\caption{Time evolution of the mesh ratio 
$R^h(t)$ for the case of 4-fold $\gamma(\theta)=1+\beta\cos(4\theta)$: (i) weak anisotropy with $\beta=0.05$; (ii) strong anisotropy with $\beta=0.3$. 
We select $h=1/160$, $\ttau=1/160$, $\sigma=-0.6$, $\eta=100$ in this test, and adopt the surface energy matrix $\mat{B}_0(\theta)$.}
\label{fig:exa2_fig1}
\end{figure}

\begin{figure}[!htp]
\centering
\includegraphics[width=0.45\textwidth]{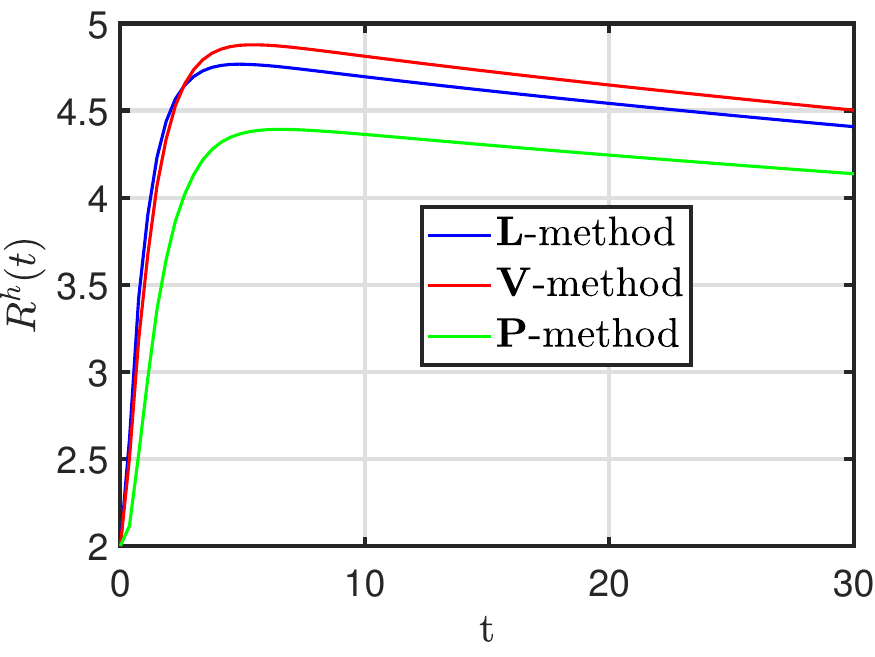}\hspace{0.5cm}
\includegraphics[width=0.45\textwidth]{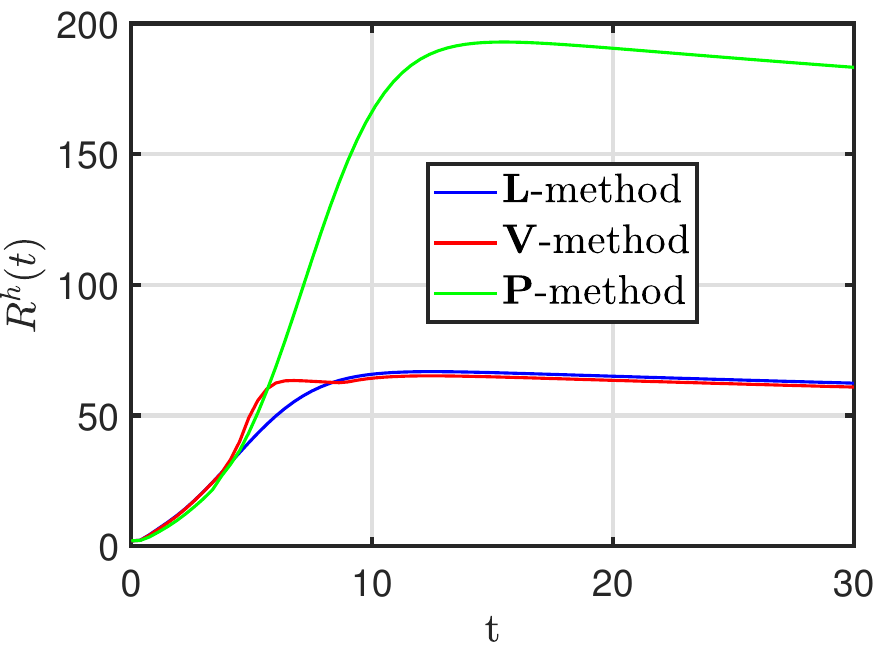}\hspace{0.5cm}
\caption{Time evolution of the mesh ratio 
$R^h(t)$ for the case of 3-fold $\gamma(\theta)=1+\beta\cos(3\theta)$: (i) weak anisotropy with $\beta=0.06$; (ii) strong anisotropy with $\beta=0.3$. 
We select 
$h=1/80$, $\ttau=1/80$, $\sigma=0.6$, $\eta=100$ in this test, and adopt the surface energy matrix $\mat{B}_1(\theta)$.}
\label{fig:exa2_fig2}
\end{figure}

We further test the volume conservation and energy stability of the $\mathbf{L}$-method, $\mathbf{V}$-method and $\mathbf{P}$-method. 
As illustrated in the left figure of Figure \ref{fig:exa2_fig3} with respect to the 4-fold anisotropy: $\gamma(\theta)=1+\beta\cos(4\theta)$, it can be found that  
$\mathbf{L}$-method has volume loss while the other two methods conserve volume as expected. 
In addition, observed from the right figure of Figure \ref{fig:exa2_fig3}, we notice that all three methods maintain energy stability. 
However, the energy decrease for the $\mathbf{L}$-method surpasses that of the other two schemes, possibly due to its volume loss. Similar test results can be observed in Figure \ref{fig:exa2_fig4}, considering the 3-fold anisotropy: $\gamma(\theta)=1+\beta\cos(3\theta)$.

\begin{figure}[!htp]
\centering
\includegraphics[width=0.45\textwidth]{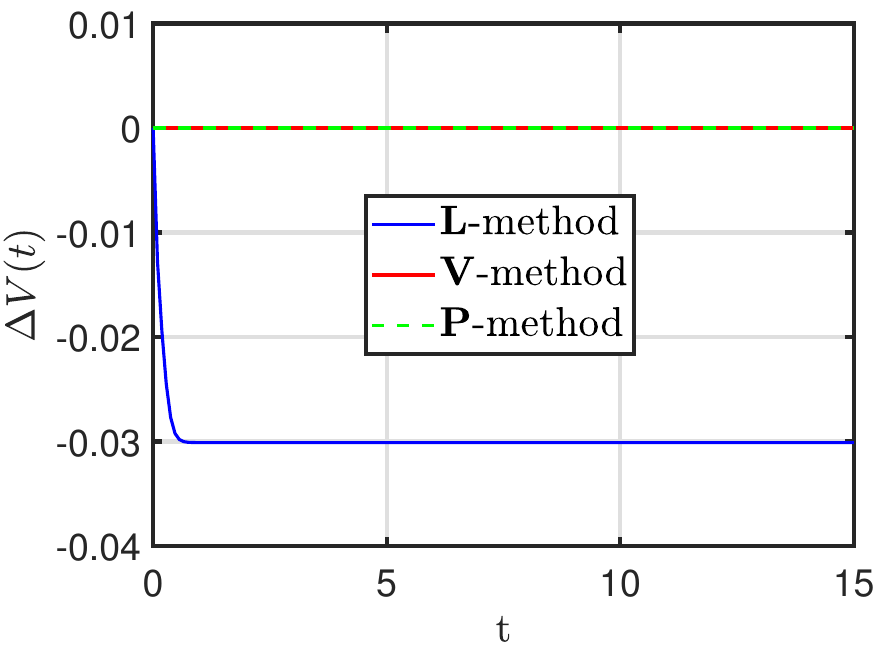}\hspace{0.5cm}
\includegraphics[width=0.45\textwidth]{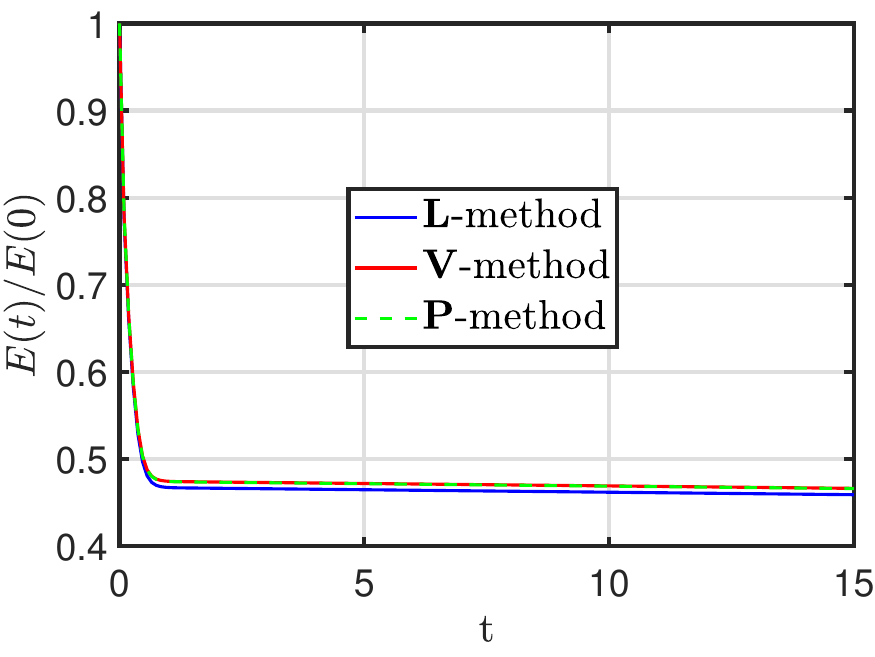}
\caption{The time evolution of the relative volume loss $\Delta V(t)$ and the energy ratio $E(t)/E(0)$ for the case of 4-fold anisotropy: $\gamma(\theta)=1+\beta\cos(4\theta)$.
We select $h=1/160$, $\ttau=1/160$, $\sigma=-0.6$, $\eta=100$, $\beta=0.3$ in this test, and adopt the surface energy matrix $\mat{B}_0(\theta)$.}
\label{fig:exa2_fig3}
\end{figure}

\begin{figure}[!htp]
\centering
\includegraphics[width=0.45\textwidth]{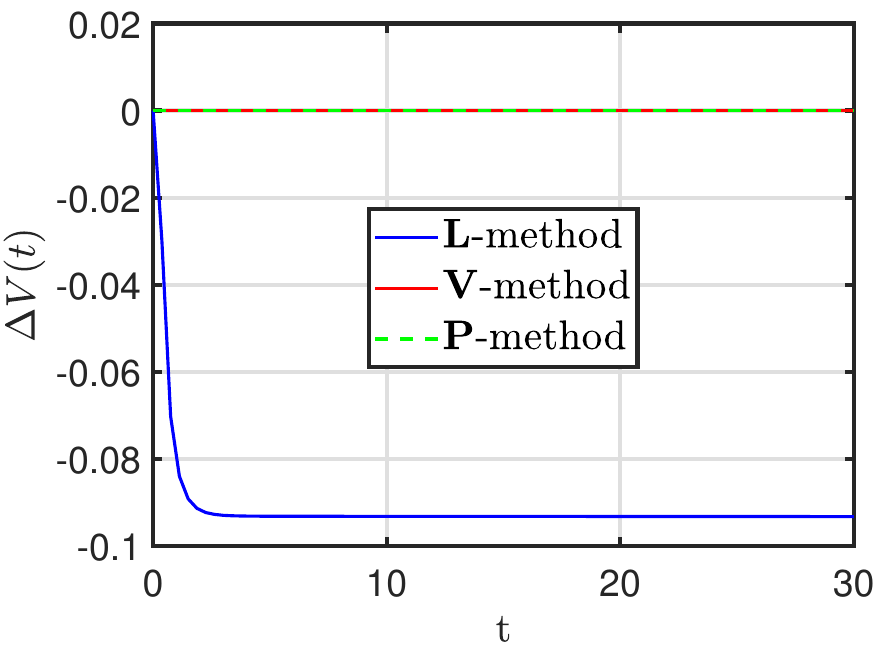}\hspace{0.5cm}
\includegraphics[width=0.45\textwidth]{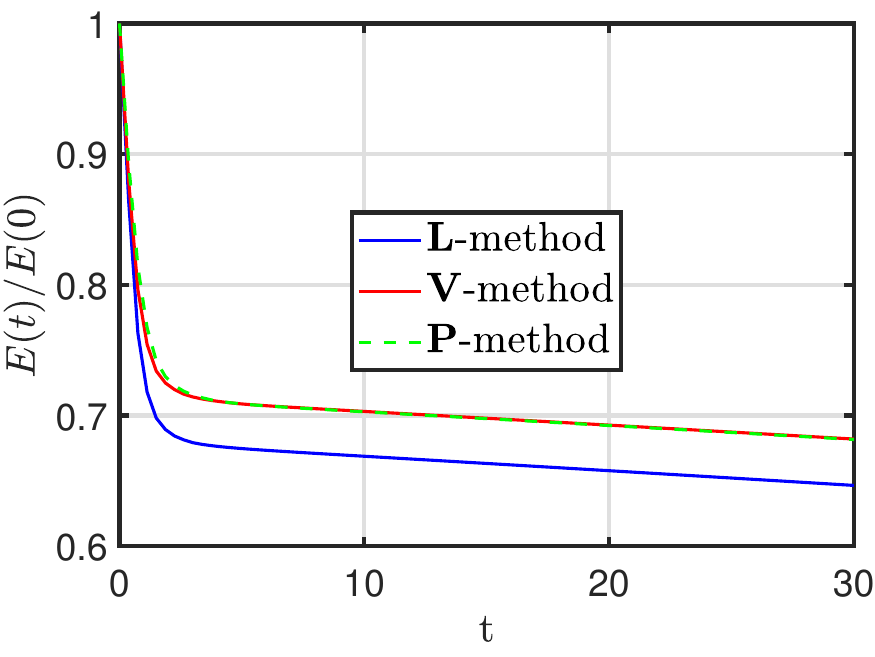}
\caption{
The time evolution of the relative volume loss $\Delta V(t)$ and the energy ratio $E(t)/E(0)$ 
 for the case of 3-fold anisotropy: $\gamma(\theta)=1+\beta\cos(3\theta)$.
We select $h=1/80$, $\ttau=1/80$, $\sigma=0.6$, $\eta=100$, $\beta=0.3$ in this test, and adopt the surface energy matrix $\mat{B}_1(\theta)$.}
\label{fig:exa2_fig4}
\end{figure}

\textbf{Example 3}:
In this example, we study the evolution of films with 'BGN' anisotropy. For 'BGN' anistropy, the surface energy matrix $\mat{B}_0(\theta)$ is given as 
\begin{equation}
\mat{B}_0(\theta)=\sum_{l}^{L}\gamma_l (\theta)^{-1}\mat{J}^{\top}\mat{G_l}\mat{J}. 
\end{equation}
In particular, we choose $L=2$ and $\mat{G_1}$, $\mat{G_2}$, $\mat{J}$ and $\gamma_l(\theta)$ ($l=1,2$) are denoted as 
\begin{equation*}
\mat{G_1}=\begin{pmatrix}
 1 & 0\\
 1 & \varepsilon ^2
\end{pmatrix}, \quad\mat{G_2}=\begin{pmatrix}
 \varepsilon ^2 & 0\\
 0 & 1
\end{pmatrix}, \quad\mat{J}=\begin{pmatrix}
 0 & 1\\
 -1 & 0
\end{pmatrix}, \quad\gamma_1(\theta)=\sqrt{\sin{\theta}^2+\varepsilon ^2\cos{\theta}^2}, \quad\gamma_2(\theta)=\sqrt{\varepsilon ^2\sin{\theta}^2+\cos{\theta}^2},  \nonumber
\end{equation*}
which the anisotropy can be represented as $\gamma(\theta)=\gamma_1(\theta)+\gamma_2(\theta)$. We choose the semi-ellipse rotation as the initial shape, with the major axis $0.66$ and minor axis $1$. The results of simulation are plotted in Figures \ref{fig:exa3_fig1} - \ref{fig:exa3_fig2}. Throughout the result, we can find energy-dissipative and volume-conservative properties during the evolution, and as the evolution of thin films the holes become smaller and smaller.
\begin{figure}[!htp]
\centering
\includegraphics[width=0.6\textwidth]{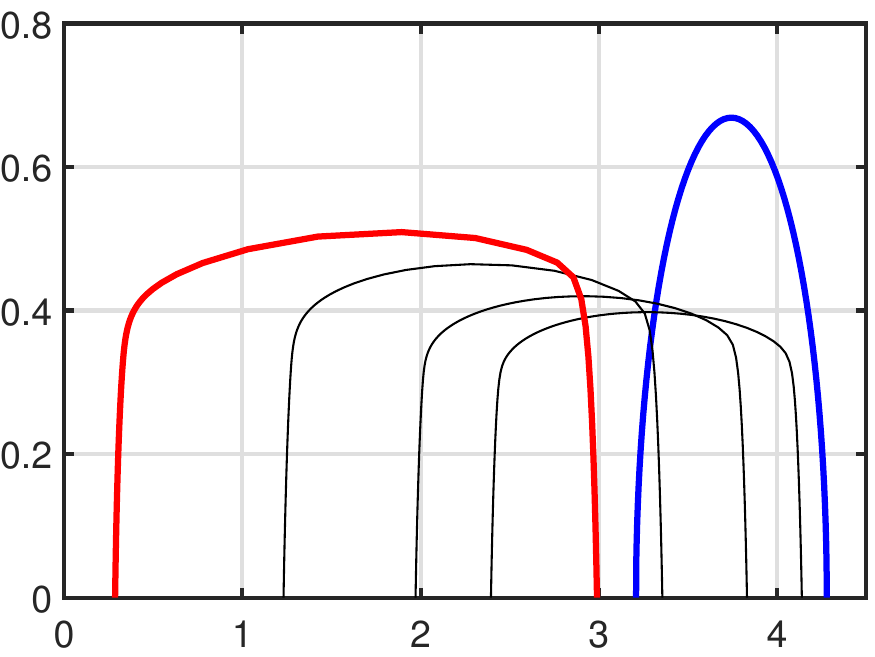}\qquad
\includegraphics[width=0.9\textwidth]{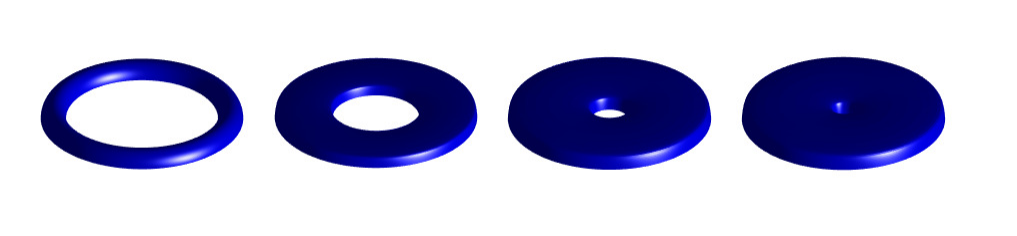}
\caption{ On the upper panel we show the generating curves $\Gamma^m$ at $t=0, 0.45, 1.05, 1.4, 1.5$. On the lower panel, we show the corresponding axisymmetric surfaces $S^m$ generated by $\Gamma^m$. Here 
$h=1/80$, $\ttau=1/100$, $\sigma=-0.6$.} 
\label{fig:exa3_fig1}
\end{figure}

\begin{figure}[!htp]
\centering
\hspace{0.5cm}
\includegraphics[width=0.45\textwidth]{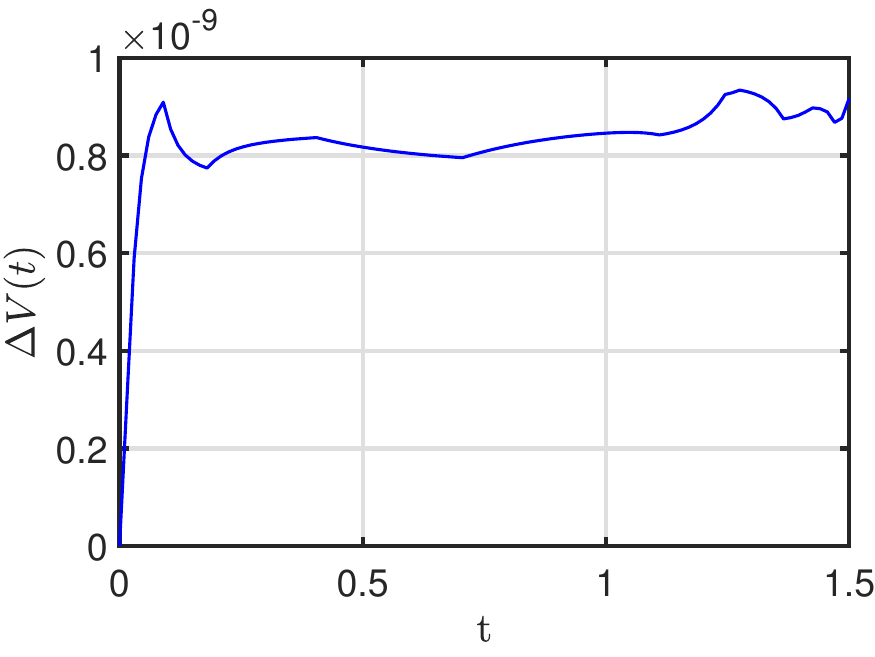}\hspace{0.5cm}
\includegraphics[width=0.45\textwidth]{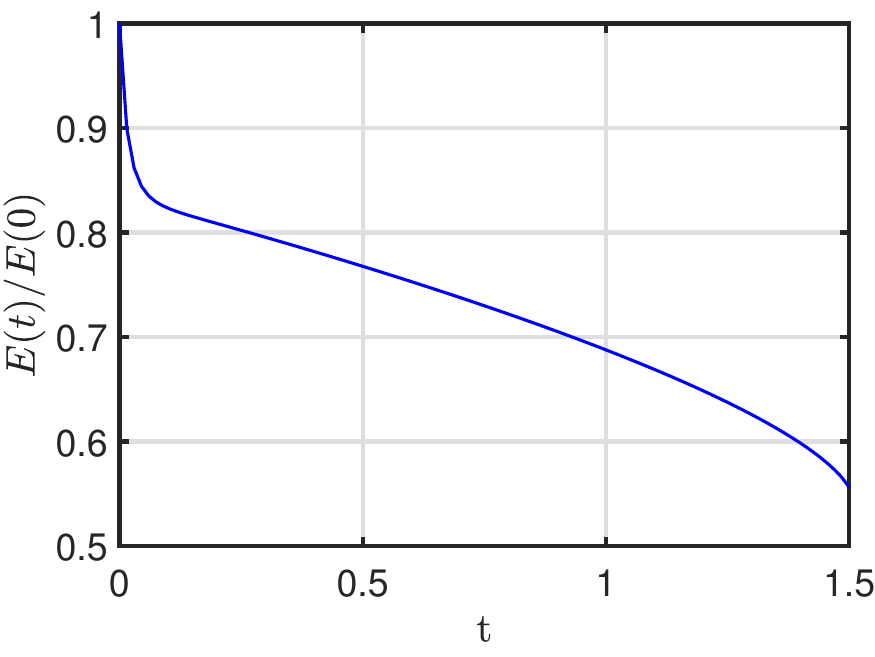}
\caption{ The time evolution of the relative volume loss $\Delta V(t)$ and the energy ratio $E(t)/E(0)$ for 'BGN' anisotropy.Here 
$h=1/80$, $\ttau=1/100$, $\sigma=-0.6$.} 
\label{fig:exa3_fig2}
\end{figure}
\textbf{Example 4}:
In this example, we consider 
the influence of different parameters for 
the evolution of films with 5-fold anisotropy: $\gamma(\theta)=1+\beta\cos(5\theta)$. We make the following two tests:
\begin{itemize}
    \item With different anisotropic parameter $\beta$, we plot the time evolution of the energy ratio, and the axisymmetric surfaces of several special moments in Figure \ref{fig:exa4_fig1}.
    \item We study the equilibrium shapes of the evolution with different values of the parameter $\sigma$. As illustrated in Figure \ref{fig:exa5_fig1}, the material constant $\sigma$ greatly influences on the equilibrium shapes. In addition, the time evolution of the relative volume loss and the energy ratio are depicted in Figure \ref{fig:exa5_fig2}.
\end{itemize}

\begin{figure}[!htp]
\centering
\includegraphics[width=0.3\textwidth]{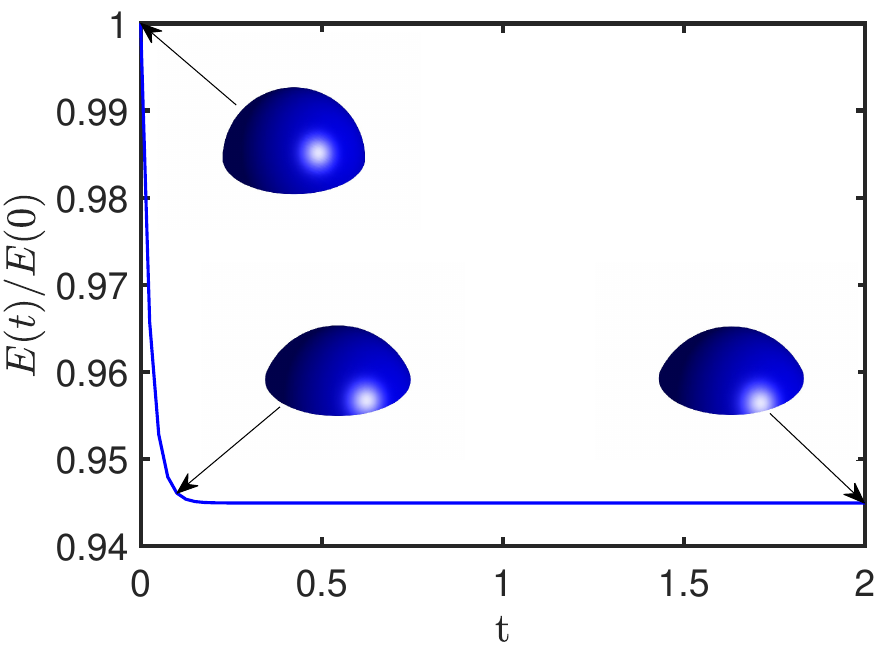}\qquad
\includegraphics[width=0.3\textwidth]{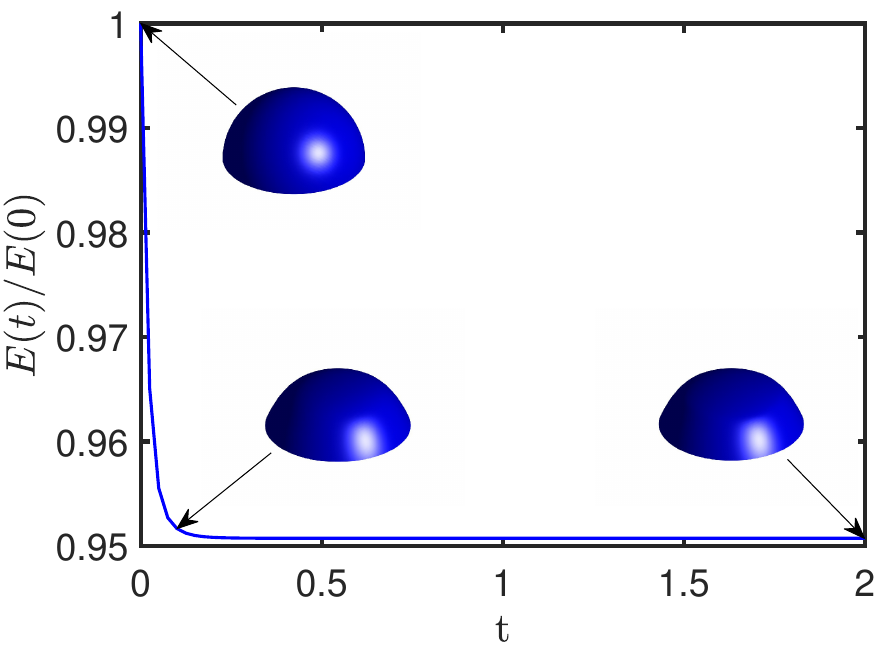}\qquad
\includegraphics[width=0.3\textwidth]{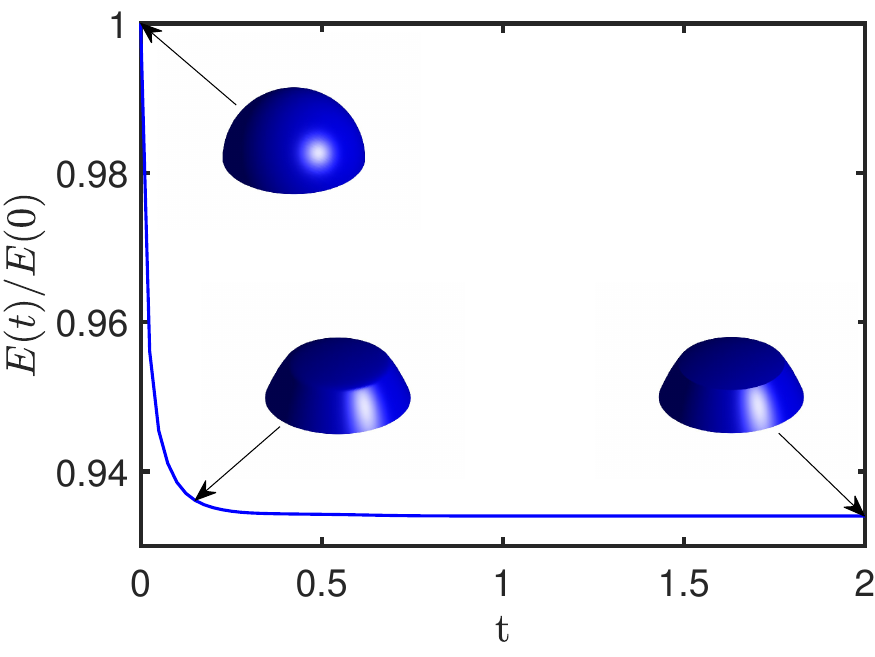}
\caption{
The time evolution of the relative the energy ratio $E(t)/E(0)$ and the axisymmetric surfaces of several special moments
 for the case of 5-fold anisotropy: $\gamma(\theta)=1+\beta\cos(5\theta)$.
The parameters are selected by $h=1/80$, $\ttau=1/40$, $\sigma=-0.4$, $\eta=100$, $\beta=0,0.03,0.07$ in this test, and we choose the semi-ellipse rotation as the initial shape, and adopt the surface energy matrix $\mat{B}_1(\theta)$.
} 
\label{fig:exa4_fig1}
\end{figure}

\begin{figure}[!htp]
\centering
\includegraphics[width=0.6\textwidth]{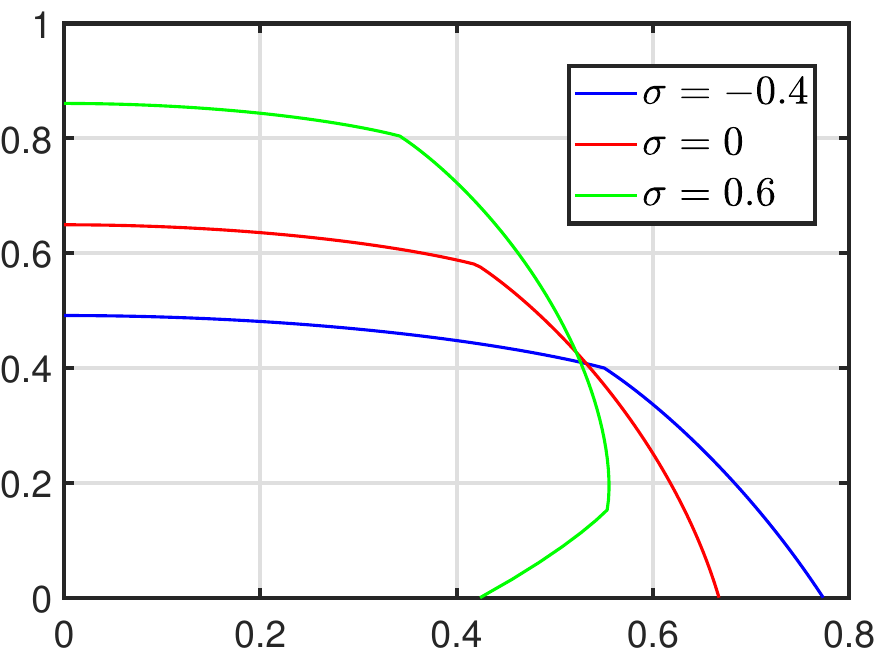}

\includegraphics[width=0.33\textwidth]{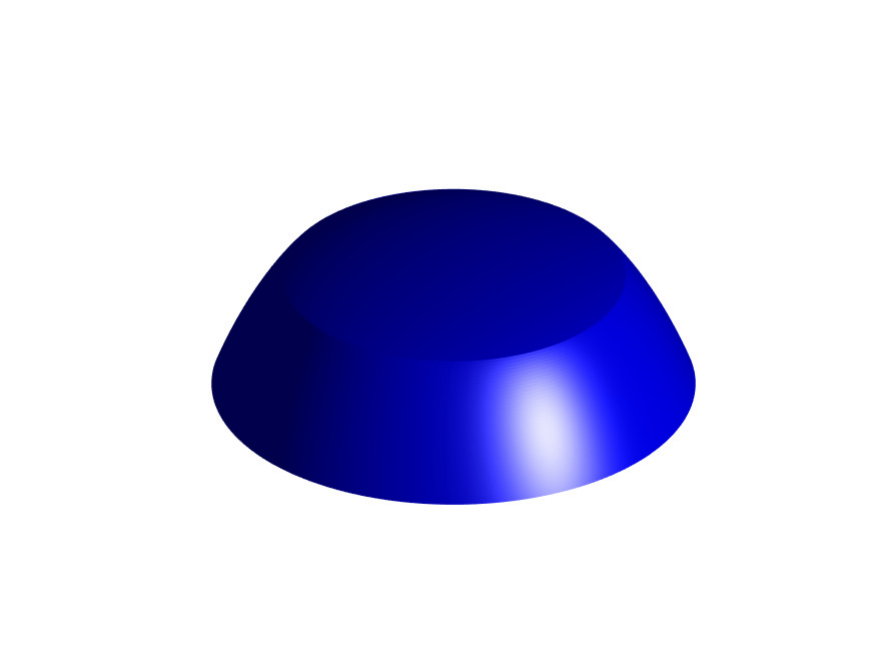}
\includegraphics[width=0.33\textwidth]{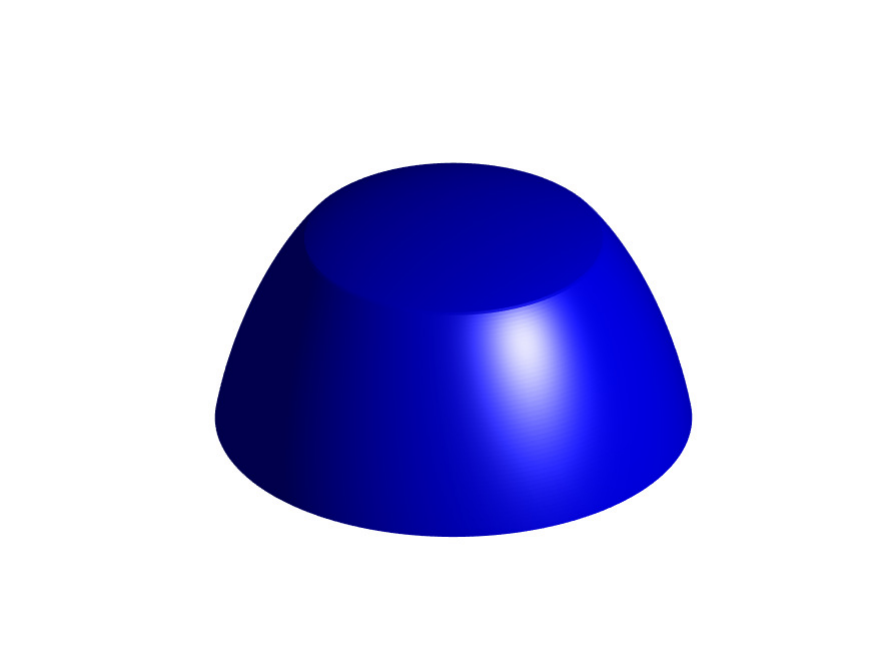}
\includegraphics[width=0.33\textwidth]{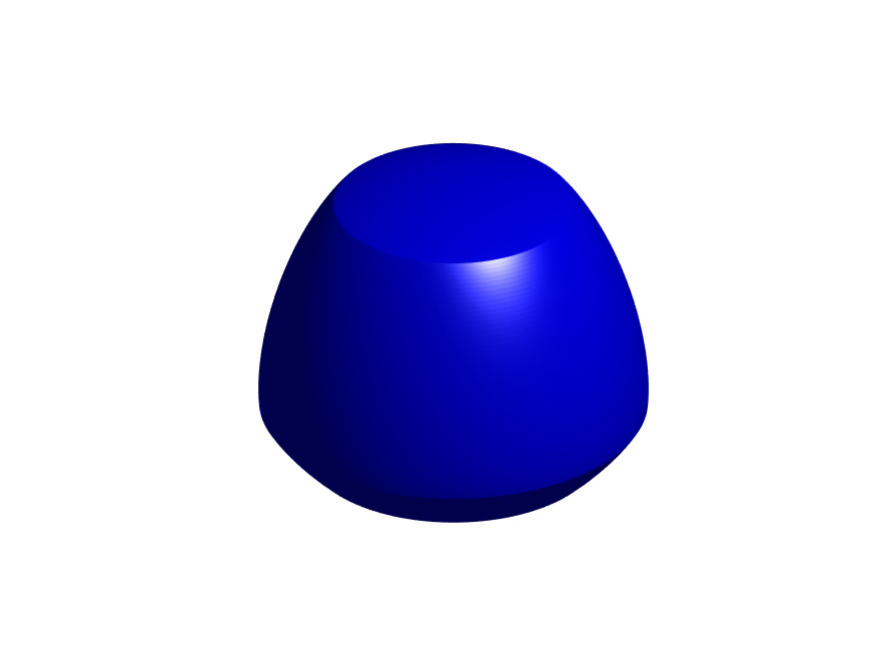}
\caption{The generating curves $\Gamma^m$ with $\sigma=-0.4,0,0.6$ in the state of equilibrium (upper pane), and the corresponding axisymmetric surfaces $S^m$ generated by $\Gamma^m$ (lower panel). Here 
$h=1/80$, $\ttau=1/40$, and we selct 5-fold anisotropy: $\gamma(\theta)=1+\beta\cos(5\theta)$.} 
\label{fig:exa5_fig1}
\end{figure}
\begin{figure}[!htp]
\centering
\hspace{0.5cm}
\includegraphics[width=0.45\textwidth]{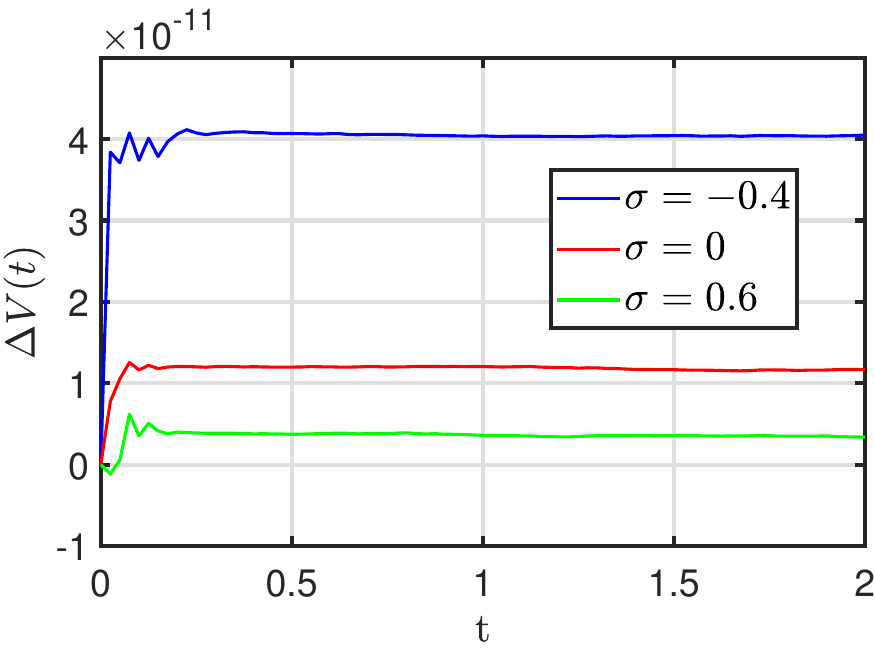}\hspace{0.5cm}
\includegraphics[width=0.45\textwidth]{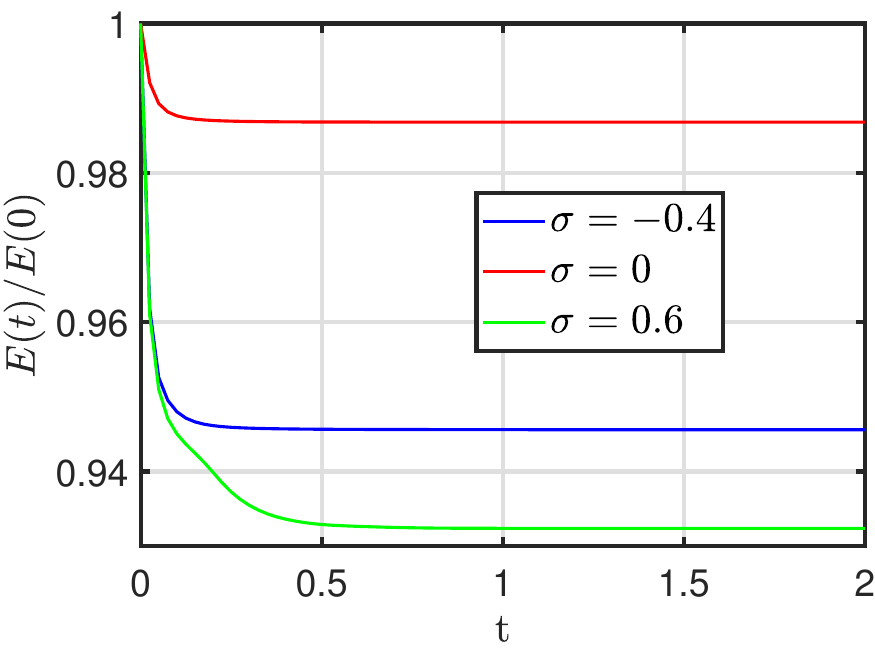}
\caption{The time evolution of the relative volume loss $\Delta V(t)$ and the energy ratio $E(t)/E(0)$ for the case of 5-fold anisotropy: $\gamma(\theta)=1+\beta\cos(5\theta)$, where $h=1/80$, $\ttau=1/40$, $\sigma=-0.4,0,0.6$.} 
\label{fig:exa5_fig2}
\end{figure}

\textbf{Example 5}:
In this concluding example, we focus on the intricate alterations that take place during the evolution of thin films. The anisotropy in this example is chosen by $\gamma(\theta)=1+\beta\cos(4\theta)$. We mainly do the following three tests:
\begin{itemize}
    \item We investigate the evolution of a thin film with initial torus. 
   As the time goes, the holes are minute enough to gradually vanish over time. Once the generating curve touches $z$-axis, by artificially updating the boundary conditions, it ultimately generates a closed pattern with distinct corners in equilibrium. Several specific moments in the evolution process are given in Figure \ref{fig:exa6_fig1}.
    \item We conduct an investigation on the progression of the elongated thin film. Our findings reveal that as time evolved, the thin film undergoes a pinch-off process, ultimately forms two separate films. The two separate films continue to evolve separately, and the one on the right eventually hits the $z$-axis, see Figures \ref{fig:exa6_fig2}-\ref{fig:exa6_fig22}.
    \item We finally study the evolution of a longer film. Unlike the previous example, the two films undergo separation and subsequently reunite, see Figures \ref{fig:exa6_fig3}-\ref{fig:exa6_fig33}. 
\end{itemize}

\begin{figure}[!htp]
\centering
\includegraphics[width=0.8\textwidth]{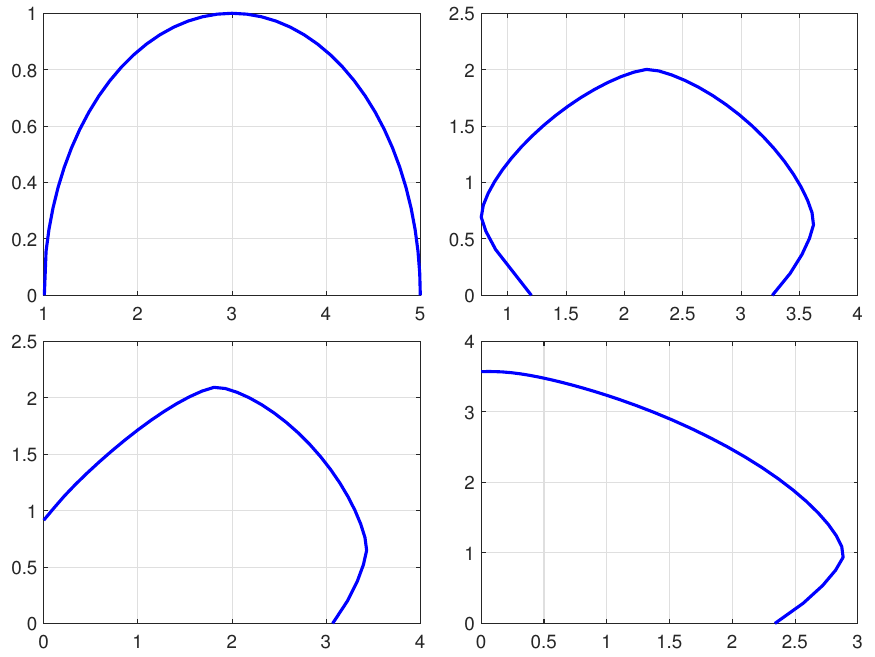}\qquad
\includegraphics[width=0.9\textwidth]{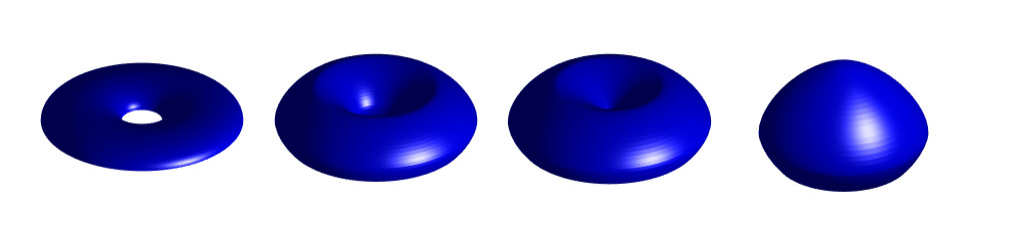}
\caption{ The generating curves $\Gamma^m$ at $t=0, 4.5, 5.5, 10$ (upper pane), and the corresponding axisymmetric surfaces $S^m$ generated by $\Gamma^m$ (lower panel). Here 
$h=1/40$, $\ttau=1/4$, $\sigma=0.6$.} 
\label{fig:exa6_fig1}
\end{figure}

\begin{figure}[!htp]
\centering
\includegraphics[width=0.8\textwidth]{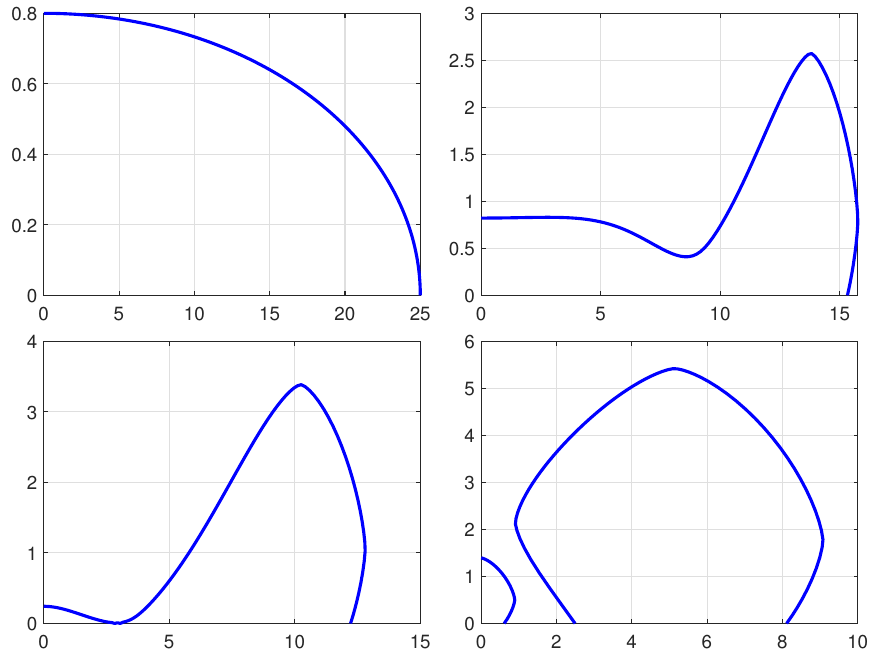}\qquad
\includegraphics[width=0.9\textwidth]{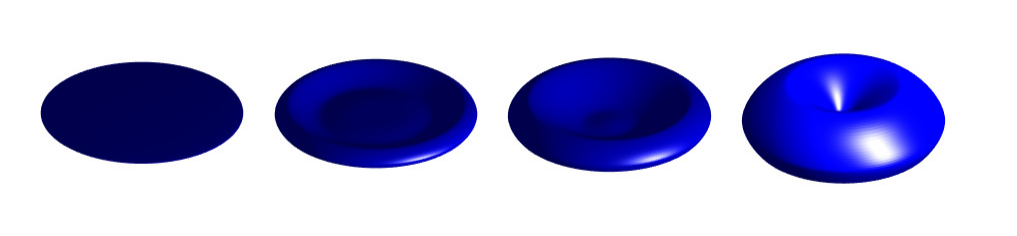}
\caption{ The generating curves $\Gamma^m$ at $t=0, 32, 86.4, 494.4$ (upper pane), and the corresponding axisymmetric surfaces $S^m$ generated by $\Gamma^m$ (lower panel). Here 
$h=1/100$, $\ttau=1.6$, $\sigma=0.6$.} 
\label{fig:exa6_fig2}
\end{figure}

\begin{figure}[!htp]
\centering
\includegraphics[width=0.45\textwidth]{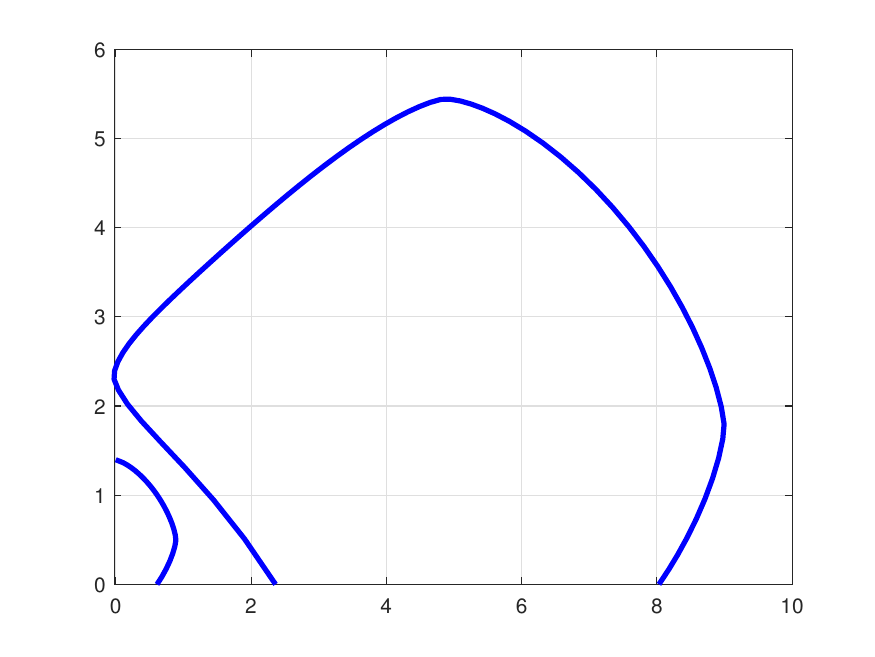}
\includegraphics[width=0.5\textwidth]{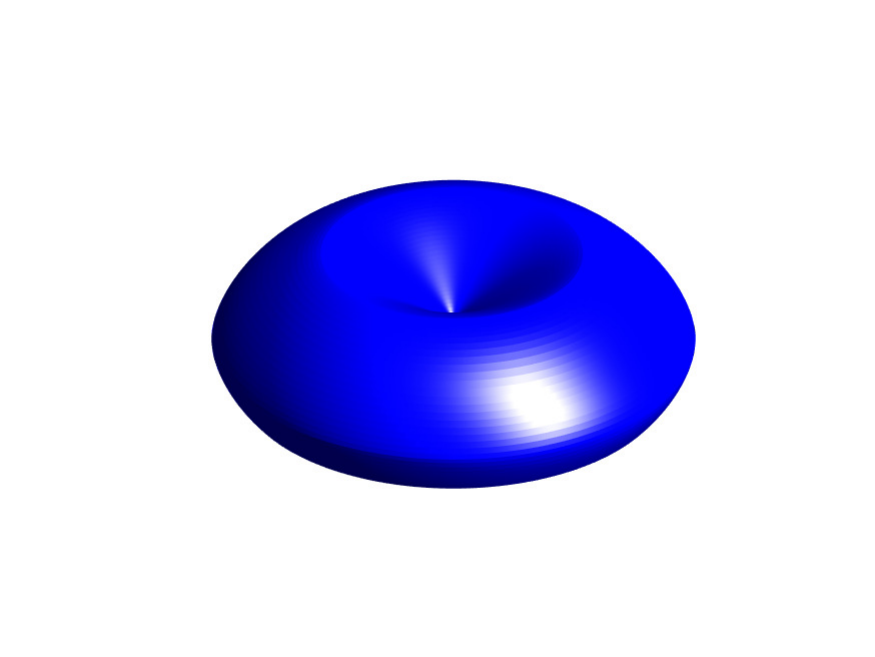}
\caption{ The generating curve $\Gamma^m$ eventually hits the $z$-axis (left panel), and the corresponding axisymmetric surfaces $S^m$ generated by $\Gamma^m$ are depicted (right panel). Here 
$h=1/100$, $\ttau=1.6$, $\sigma=0.6$.} 
\label{fig:exa6_fig22}
\end{figure}

\begin{figure}[!htp]
\centering
\includegraphics[width=0.8\textwidth]{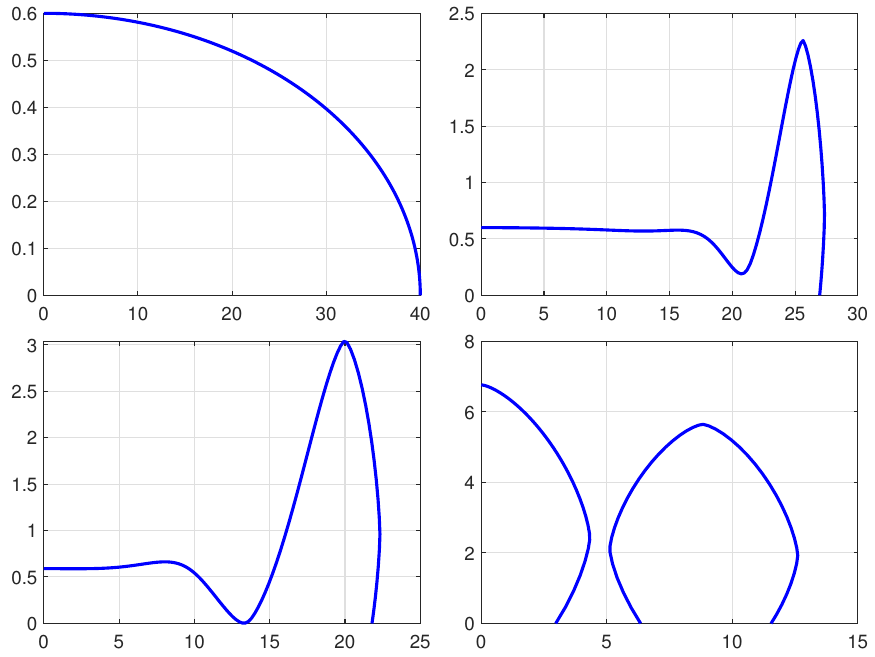}\qquad
\includegraphics[width=0.9\textwidth]{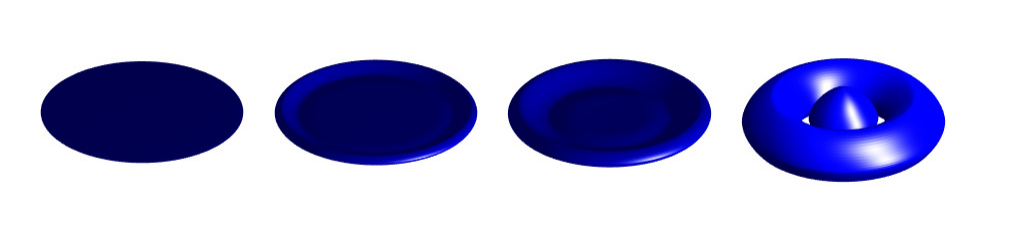}
\caption{ The generating curves $\Gamma^m$ at $t=0, 32, 163.8, 2000$ (upper pane), and the corresponding axisymmetric surfaces $S^m$ generated by $\Gamma^m$ (lower panel). Here 
$h=1/100$, $\ttau=1.6$, $\sigma=0.6$.} 
\label{fig:exa6_fig3}
\end{figure}

\begin{figure}[!htp]
\centering
\includegraphics[width=0.45\textwidth]{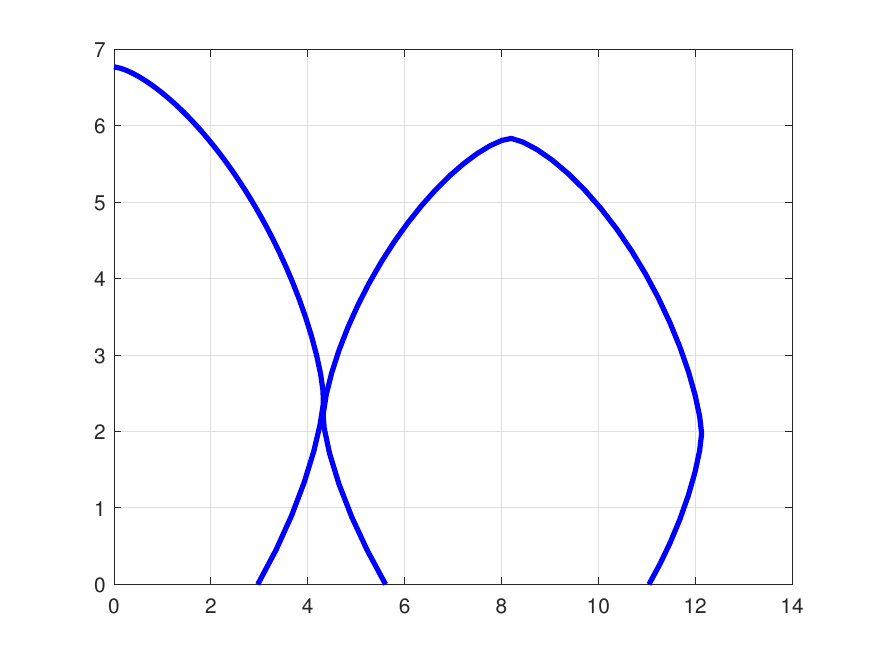}
\includegraphics[width=0.5\textwidth]{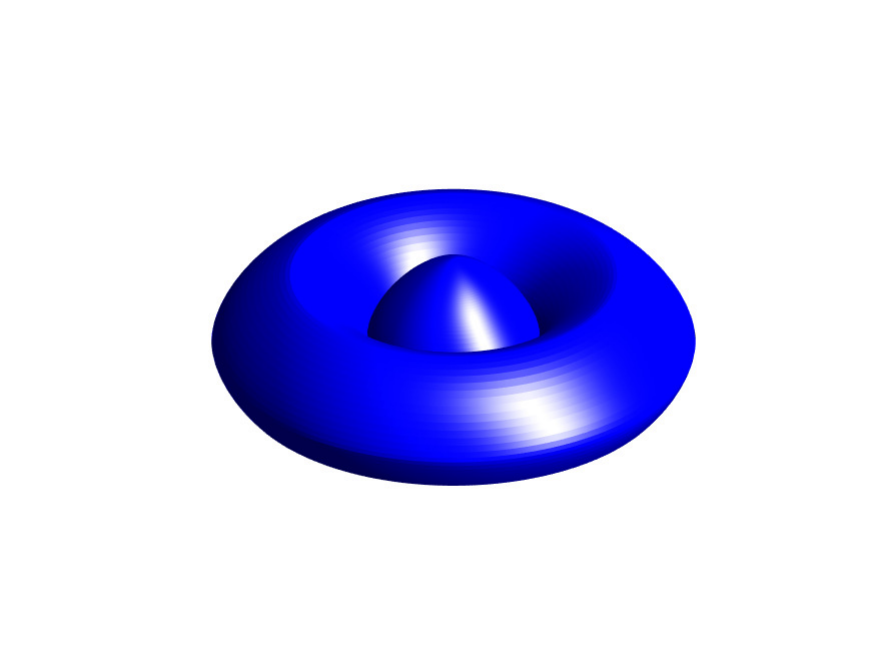}
\caption{ The separated curves eventually touch again, and the corresponding axisymmetric surfaces $S^m$ generated by $\Gamma^m$ are depicted (right panel). Here 
$h=1/100$, $\ttau=1.6$, $\sigma=0.6$.} 
\label{fig:exa6_fig33}
\end{figure}

\section{Conclusions}\label{sec7}
In this work, we focus on the efficient PFEMs for the axisymmetric SSD 
 with anisotropic surface energy. 
 Through the introduction of two types of surface energy matrices with respect to the orientation angle $\theta$, we develop two structure-preserving algorithms for the axisymmetric SSD, which exhibit applicability across a wider range of anisotropy functions and are theoretically proven to uphold volume conservation and energy stability.
Moreover, leveraging a novel weak formulation for axisymmetric SSD, we construct another two numerical schemes with relatively good mesh quality. Through numerous numerical tests, we have showcased the accuracy, structure preservation, and efficiency of our numerical methods. 
\bibliographystyle{elsarticle-num}
\bibliography{thebib}
\end{document}